\documentclass[amsthm]{elsart}

\usepackage{yjsco}
\usepackage{natbib}
\usepackage{color}
\usepackage{fancyvrb,shortvrb}
\usepackage{amssymb}
\usepackage{amsmath}
\input{tcilatex}
\input xypic
\input xy
\xyoption{all}

\newtheorem{examples}[example]{Examples}

\begin{document}

\date{2 April 2014}

\begin{frontmatter}

\title{Isoclinic Crossed Modules with GAP Implementations }

\thanks{This research was partly supported by .....}

\author{A. Odaba\c{s}}
\address{Osmangazi University \\ Department of
Mathematics - Computer \\Eski\c{s}ehir-Turkey}
\ead{aodabas@ogu.edu.tr}
\ead[url]{http://fef.ogu.edu.tr/matbil/aodabas/}

\author{E.\"{O}. Uslu}
\address{Osmangazi University \\ Department of
Mathematics - Computer \\Eski\c{s}ehir-Turkey}
\ead{enveruslu@ogu.edu.tr}
\ead[url]{http://fef.ogu.edu.tr/matbil/enveruslu/}

\author{E. Ilgaz}
\address{Osmangazi University \\ Department of
Mathematics - Computer \\Eski\c{s}ehir-Turkey}
\ead{eilgaz@ogu.edu.tr}
\ead[url]{http://fef.ogu.edu.tr/matbil/eilgaz/}

\begin{abstract}
We introduce the isoclinism of crossed modules. We also give GAP
implementations for constructing the isoclinism families of finite crossed
modules and consequently give an enumeration about isoclinic crossed modules
existing in the GAP library. 
\end{abstract}


\end{frontmatter}

\section{Introduction}
The notion of isoclinism was first defined in \cite{PH}, for a
classification of finite groups whose orders are prime powers. This work was
detailed in \cite{MH}. After then, some new results
were given in many papers, such as \cite{MRJ,RM,HM,FP,ARS,JT}. Also the
relation between groups and their stem extensions with Schur multiplicators
were given in \cite{TB} where they consider the isoclinism
of central group extensions. In this work we consider the 2-dimensional
group version, \ called \textquotedblleft crossed module\textquotedblright
,\ of isoclinism and determine some basic results. This construction gives a
new equivalence relation on crossed modules weaker then isomorphism. So we
have a new classification of crossed modules and the resulting equivalence
classes, called isoclinism families, are convenient with some algebraic
invariants such as nilpotency classes, rank and middle length. For determine
such a comparison, we defined some new concepts for crossed modules such as;
class preserving actor, rank and middle length of crossed modules etc.

In general, the isoclinism is used for classification of finite groups, and
there are many works concerning the enumeration of groups with finite order
related to isoclinism, \cite{MH,RJ}. So, we first construct GAP
implementation for isoclinic groups, and as an example we give a character
table consisting of the isoclinism families of the group with order 192.
Enumeration of crossed modules and related structures are determined from
many view points. In \cite{AW,RB,GE,GE1,GE2}, one can found many
computations about these notions with GAP. By a similar way, we give a
GAP implementation for classification of finite crossed modules up to
isoclinism. For this, we added some new functions which do not exist in XMod
package, such as; \texttt{DerivedSubXMod(XM)}, \texttt{FactorXMod(XM,NM)},
\texttt{IsIsoclinicXMod(XM1,XM2)}.

In order to get our goals, we organize the paper as follows;

In Section 2, we recall some needed results about crossed modules and
introduce some new notions which will be used in the sequel of the paper.

In Section 3, we introduce the notion of isoclinism for crossed modules and
establish the basic theory. As expected, we give the compatibility of this
definition with nilpotency, solvability and class preserving actors of
crossed modules as it was the case for groups.

In Section 4, we give GAP implementations for computing the isoclinism
families of crossed modules in low order.

In Section 5, by using these implementations, we give character tables
consisting isoclinism families of certain crossed modules. These tables are
particular examples which show that the definition of isoclinism is convenient
with the nilpotency classes and other invariants in low order.

\section{Preliminaries}

In this section we recall some needed material about crossed modules.
Crossed modules were defined by \cite{WH}, as an
algebraic model for homotopy 3-types. We refer to \cite{RB2, NO, NO1, TP2},
for a compherensive and detailed work.

\begin{definition}
\textit{A crossed module} is a group homomorphism%
\begin{equation*}
d:G_{1}\longrightarrow G_{0}
\end{equation*}%
with a left action from $G_{0}$ on $G_{1}$ written $(g_{0},g_{1})\mapsto $ $%
^{g_{0}}g_{1},$ for $g_{0}\in G_{0},g_{1}\in G_{1}$ satisfying the following
conditions:%
\begin{equation*}
\begin{array}{llll}
1) & d(^{g_{0}}g_{1}) & = & g_{0}d(g_{1})g_{0}^{-1}, \\
2) & ^{d(g_{1})}g_{1}^{\prime } & = & g_{1}g_{1}^{\prime }g_{1}^{-1},%
\end{array}%
\end{equation*}%
for all $g_{0}\in G_{0},$ $g_{1},g_{1}^{\prime }\in G_{1}.$
\end{definition}

We will denote such a crossed module by $G:G_{1}\overset{d}{\longrightarrow }%
G_{0}.$

\begin{examples}
\ \ \ \ \
\end{examples}

(1) Let $N$ be a normal subgroup of $M$. Then {\xymatrix {N
\ar@{^{(}->}[r]^{inc.} &M}} is a crossed module with conjugate action of $M$
on $N$. Consequently, every group $M$ can be thought as a crossed module in
the two obvious way: {\xymatrix {1 \ar@{^{(}->}[r]^{inc.} &M}} or $M\overset{%
id}{\longrightarrow }M.$

(2) $1\overset{1}{\longrightarrow }1$ is a crossed module and it is called
the \textit{trivial crossed module}.

(3) $K\overset{1}{\longrightarrow }L$ is a crossed module, where $K$ is a $L$%
-module and the boundary operator is the zero map.

(4) $M\overset{c}{\longrightarrow }Aut(M)$ is a crossed module, where $c$
assigns to each element $x\in M$, the inner automorphism $c_{x}$ of $M$
defined by $c_{x}(m)=xmx^{-1},$ for all $m\in M.$

\begin{definition}
A crossed module $G:G_{1}\overset{d}{\longrightarrow }G_{0}$ is called
\textit{aspherical} if ker$d=1,$ i.e $d$ is injective, and \textit{simply
connected} if coker$d=1,$ i.e $d$ is surjective.
\end{definition}

\textit{A morphism} between two crossed modules $G:G_{1}\overset{d}{%
\longrightarrow }G_{0}$ and $G^{\prime }:G_{1}^{\prime }\overset{d^{\prime }}%
{\longrightarrow }G_{0}^{\prime }$ is a pair $(\alpha ,\beta )$ of group
homomorphisms $\alpha :G_{1}\longrightarrow G_{1}^{\prime },$ $\beta
:G_{0}\longrightarrow G_{0}^{\prime },$ such that $\beta d=d^{\prime }\alpha
$ and $\alpha (^{g_{0}}g_{1})=$ $^{\beta (g_{0})}\alpha (g_{1}),$ for all $%
g_{0}\in G_{0},$ $g_{1}\in G_{1}$. Consequently we have the category $%
\mathbf{XMod}$ whose objects are the crossed modules and its morphisms are
the morphisms of crossed modules.

\begin{remark}
Since $M\overset{id}{\longrightarrow }M$ is a crossed module, for any group $%
M$, the category of groups can be thought as a full subcategory of crossed
modules.
\end{remark}

We say that $(\alpha ,\beta )$ is an automorphism of $G$ if $\alpha $ and $%
\beta $ are both automorphisms. We denote the group of automorphisms of the
crossed module $G:G_{1}\overset{d}{\longrightarrow }G_{0}$ by $Aut(G)$ where
the multiplication is defined by componentwise composition.

Existing of zero object and equalizers give rise to subobjects and normal
subobjects. A crossed module $H:H_{1}\overset{d_{H}}{\longrightarrow }H_{0}$
is a \textit{subcrossed module} of a crossed module $G:G_{1}\overset{d_{G}}{%
\longrightarrow }G_{0}$ if $H_{1},$ $H_{0}$ are subgroups of $G_{1}$, $%
G_{0}, $ respectively, $d_{H}=d_{G}|_{H_{1}}$ and the action of $H_{0}$ on $%
H_{1}$ is induced by the action of $G_{0}$ on $G_{1}.$ Also, a subcrossed
module $H:H_{1}\overset{d_{H}}{\longrightarrow }H_{0}$ of a crossed module $%
G:G_{1}\overset{d_{G}}{\longrightarrow }G_{0}$ is \textit{normal} if $H_{0}$
is a normal subgroup of $G_{0},$ $^{g_{0}}h_{1}\in H_{1}$ and $%
^{h_{0}}g_{1}g_{1}^{-1}\in H_{1},$ for all $g_{0}\in G_{0},$ $g_{1}\in
G_{1}, $ $h_{0}\in H_{0},$ $h_{1}\in H_{1}.$ Consequently, we have the
\textit{quotient crossed module} $G/H:G_{1}/H_{1}\overset{\overline{d_{G}}}{%
\longrightarrow }G_{0}/H_{0}$ with the induced boundary map and action.

\begin{definition}
Let $(\alpha ,\beta ):(G:G_{1}\overset{d}{\longrightarrow }%
G_{0})\longrightarrow (G^{\prime }:G_{1}^{\prime }\overset{d^{\prime }}{%
\longrightarrow }G_{0}^{\prime })$ be a crossed mo\-dule morphism. The
\textit{kernel} of $(\alpha ,\beta )$ is the normal subcrossed module $(\ker
\alpha ,$ $\ker \beta ,$ $d|)$ of $G,$ denoted by $\ker (\alpha ,\beta )$
and the \textit{image} $\func{Im}(\alpha ,\beta )$ is the subcrossed module $%
\left( \func{Im}\alpha ,\func{Im}\beta ,d^{\prime }|\right) $ of $G^{\prime
}.$
\end{definition}

Analogous to group theory, we have the third isomorphism theorem for crossed
modules given in \cite{NO1}.

Let $H:H_{1}\overset{d}{\longrightarrow }H_{0}$ and $K:K_{1}\overset{d}{%
\longrightarrow }K_{0}$ be a subcrossed modules of $G:G_{1}\overset{d}{%
\longrightarrow }G_{0}.$ Then the \textit{intersection} of $H$ and $K$ is
defined as the subcrossed module%
\begin{equation*}
H\cap K:H_{1}\cap K_{1}\overset{d}{\longrightarrow }H_{0}\cap K_{0},
\end{equation*}%
which is normal in $G.$ $HK$ is also defined as the crossed submodule $%
HK:H_{1}K_{1}\overset{d}{\longrightarrow }H_{0}K_{0}$ when $K$ is normal.
Consequently, we have

\begin{equation*}
\frac{H}{H\cap K}\cong \frac{HK}{K}.
\end{equation*}

\begin{definition}
Let $G:G_{1}\overset{d}{\longrightarrow }G_{0}$ be a crossed module. A
\textit{derivation} from $G_{0}$ to $G_{1}$ is the map $\partial
:G_{0}\longrightarrow G_{1}$ such that $\partial (xy)=\partial
(x)^{x}\partial (y),$ for all $x,y\in G_{0}.$
\end{definition}

\noindent The set of all derivations is denoted by $Der(G_{0},G_{1}).$

As defined in \cite{WH}, $Der(G_{0},G_{1})$ is a semigroup with
the multiplication $\partial _{1}\circ \partial _{2}$ defined by%
\begin{equation*}
(\partial _{1}\circ \partial _{2})(g_{0})=\partial _{1}(d\partial
_{2}(g_{0})g_{0})\partial _{2}(g_{0})=(\partial _{1}d(g_{1})g_{1})\partial
_{2}(g_{0})\partial _{1}(g_{0}),
\end{equation*}
for all $g_{0}\in G_{0},$ $\partial _{1},$ $\partial _{2}\in
Der(G_{0},G_{1}) $ where the identity element is the zero map. The \textit{%
Whitehead group} $D(G_{0},G_{1})$ is defined to be the group of units of $%
Der(G_{0},G_{1}),$ and the elements of $D(G_{0},G_{1})$ are called regular
derivations.

Due to \cite{NO}, for a given crossed module $G:G_{1}\overset{d}{%
\longrightarrow }G_{0},$ we have the crossed module
\begin{equation*}
\begin{array}{cccc}
\Delta : & D(G_{0},G_{1}) & \longrightarrow & Aut(G) \\
& \partial & \longmapsto & (\partial d,d\partial )%
\end{array}%
\end{equation*}%
with the action of $Aut(G)$ on $D(G_{0},G_{1})$ given by $^{(\alpha
,\beta )}\partial =\alpha \partial \beta ^{-1}$, for all $\left( \alpha
,\beta \right) \in Aut(G),$ $\partial \in D(G_{0},G_{1}).$ This crossed
module is called the \textit{actor }of $G$, and denoted by $Act(G).$

This structure was introduced by \cite{LU} and developed in
\cite{NO}. The actor objects are defined for introduce the actions in the
category of crossed modules from which the objects such as centers,
commutators, Abelian objects, semi direct products are defined.

There is a canonical morphism of crossed modules%
\begin{equation*}
(\eta ,\gamma ):G\longrightarrow Act(G)
\end{equation*}%
given by $\eta :G_{1}\longrightarrow D(G_{0},G_{1}),$ $\eta
_{g_{1}}(g_{0})=g_{1}$ $^{g_{0}}g_{1}^{-1}$ and $\gamma
:G_{0}\longrightarrow Aut(G),$ $\gamma (g_{0})=(\alpha _{g_{0}},\phi
_{g_{0}})$ such that $\alpha _{g_{0}}(g_{1})=$ $^{g_{0}}g_{1}$ and $\phi
_{g_{0}}(g_{0}^{\prime })=g_{0}g_{0}^{\prime }g_{0}^{-1},$ for all $%
g_{0},g_{0}^{\prime }\in G_{0},g_{1}\in G_{1}.$ The image of the morphism $%
(\eta ,\gamma )$ is called the \textit{inner actor} of the crossed module $%
G:G_{1}\overset{d}{\longrightarrow }G_{0},$ denoted by $InnAct(G).$

\begin{definition}
\label{09} The\textit{\ center} of the crossed module $G:G_{1}\overset{d}{%
\longrightarrow }G_{0}$ is defined as the normal subcrossed module%
\begin{equation*}
Z(G):G_{1}^{G_{0}}\overset{d|}{\longrightarrow }St_{G_{0}}(G_{1})\cap
Z(G_{0})
\end{equation*}%
where%
\begin{equation*}
G_{1}^{G_{0}}=\{g_{1}\in G_{1}:\text{ }^{g_{0}}g_{1}=g_{1},\text{ for all }%
g_{0}\in G_{0}\},
\end{equation*}%
\begin{equation*}
St_{G_{0}}(G_{1})=\{g_{0}\in G_{0}:\text{ }^{g_{0}}g_{1}=g_{1},\text{ for
all }g_{1}\in G_{1}\}
\end{equation*}%
and $Z(G_{0})$ is the center of $G_{0}.$This definition recovers the
generalized definition of Huq given in \cite{HU}. (Here $G_{1}^{G_{0}}$ is
called \textit{fixed point subgroup} of $G_{1}$ and $St_{G_{0}}(G_{1})$ is
called\textit{\ the stabilizer of }$G_{1}$ in $G_{0}.$) So, $G:G_{1}\overset{%
d}{\longrightarrow }G_{0}$ is called \textit{Abelian} when $G=Z(G).$
\end{definition}

\begin{definition}
\label{11} Let $G:G_{1}\overset{d}{\longrightarrow }G_{0}$ be a crossed
module. The \textit{commutator subcrossed module} $[G,G]$ of $G$ is defined
by%
\begin{equation*}
\lbrack G,G]:D_{G_{0}}(G_{1})\overset{d|}{\longrightarrow }[G_{0},G_{0}]
\end{equation*}%
where $D_{G_{0}}(G_{1})$ is the subgroup generated by $\
\{^{g_{0}}g_{1}g_{1}^{-1}:g_{1}\in G_{1},$ $g_{0}\in G_{0}\}$ and $%
[G_{0},G_{0}]$ is the commutator subgroup of $G_{0}$.
\end{definition}

\begin{proposition}
\label{10} Let $G:G_{1}\overset{d}{\longrightarrow }G_{0}$ be a crossed
module. Then we have the following;\newline
(i) If $G$ is simply connected, then $G_{1}^{G_{0}}=Z(G_{1})$ and $%
D_{G_{0}}(G_{1})=[G_{1},G_{1}]$.\newline
(ii) If $G$ is aspherical, then $Z(G_{0})=St_{G_{0}}(G_{1})\cap Z(G_{0})$.
\end{proposition}

\begin{proof}
Can be checked by a direct calculation.
\end{proof}

Finally, we have stem crossed modules defined as the crossed modules whose
centers are in their commutators. This definition was introduced in \cite{VC}%
.

\begin{example}
The crossed module $G:Kl_{4}\overset{d}{\longrightarrow }C_{3}$ is a stem
crossed module. Since $Z(G)=\{e\}\longrightarrow \{e\}$ and $%
[G,G]:Kl_{4}\longrightarrow \{e\},$ we have $Z(G)\subseteq \lbrack G,G].$
\end{example}

\begin{definition}
A crossed module $G:G_{1}\overset{d}{\longrightarrow }G_{0}$ is called
\textit{\ finite }if\textit{\ }$G_{1}$ and $G_{0}$ are finite groups.
\end{definition}

The order of a finite crossed module is defined as the pair $[m,n]$ where $%
m, $ $n$ are the orders of $G_{1},$ $G_{0},$ respectively.

Let $H:H_{1}\longrightarrow H_{0}$ be a subcrossed module of the crossed
mo\-dule $G:G_{1}\longrightarrow G_{0}.$ Suppose that there is a finite
sequence $H^{i}:(H_{1}^{i}\longrightarrow H_{0}^{i})_{0\leq i\leq n}$ of the
subcrossed modules of $G$ such that%
\begin{equation*}
H=H^{0}\trianglelefteq H^{1}\trianglelefteq \ldots \trianglelefteq
H^{n-1}\trianglelefteq H^{n}=G.
\end{equation*}%
This will be called a \textit{series} of length $n$ from $H$ to $G.$ The
subcrossed modules $H^{0},$ $H^{1},$\ldots $,$ $H^{n}$ are called the
\textit{terms} of the series and quotient crossed modules $H^{i}/H^{i-1},$ $%
i=1,\ldots ,n,$ are called the \textit{factors} of the series. A series $1$
to $G$ is shortly called a \textit{series} of $G.$ A series is called
\textit{central} if all factors are central. $G$ is called \textit{solvable}
if it has a series all of whose factors are Abelian crossed modules and is
called \textit{nilpotent} if it has a series all of whose factors are
central factors of $G.$

\begin{definition}
Let $G:G_{1}\longrightarrow G_{0}$ be nilpotent. Then, for any central
series
\begin{equation*}
1=G^{0}\trianglelefteq G^{1}\trianglelefteq \cdots \trianglelefteq
G^{r}=G_{0}
\end{equation*}%
of $G$, we have
\begin{equation*}
\Gamma _{r-i+1}(G)\leq G^{i}\leq \xi _{i}(G),
\end{equation*}

$i=0,1,\ldots ,r$ where $\Gamma _{1}(G)=(G)$, $\xi _{0}(G)=1$ and%
\begin{equation*}
\Gamma _{n}(G)=[\Gamma _{n-1}(G),G],\text{ }n>1
\end{equation*}%
\begin{equation*}
\xi _{n}(G_{0})/\xi _{n-1}(G_{0})=\xi (G_{0}/\xi _{n-1}(G_{0}))\text{, }n>0.
\end{equation*}%
Furthermore, the least integer $c$ such that $\Gamma _{c+1}(G)=1$ is equal
to the least integer $c$ such that $\xi _{c}(G)=G$. The integer $c$ is
called the \textit{nilpotency class }of the crossed module $G.$
\end{definition}

\begin{definition}
Let $G:G_{1}\longrightarrow G_{0}$ be solvable and let $n$ be the least
integer such that $G^{n}=1$ where $G^{n}$ is the subcrossed module of $G$ such
that $G^{0}=G$ and for each integer $n>0$, $G^{n}=[G^{n-1},G^{n-1}].$ Then
we call $n$ the \textit{derived length} of $G.$
\end{definition}

\section{Isoclinic Crossed Modules}

In this section we introduce the notion of isoclinic crossed modules to have
a new equivalence relation on crossed modules weaker then isomorphism, which
gives rise to a new classification. First, we recall the definition of
isoclinic groups from \cite{PH}.

Let $M$ and $N$ be two groups. $M$ and $N$ are \textit{isoclinic} if there
exist isomorphisms $\eta :M/Z(M)\longrightarrow N/Z(N)$ and $\xi
:[M,M]\longrightarrow \lbrack N,N]$ between central quotients and derived
subgroups, respectively, such that, the following diagram is commutative:

\begin{equation*}
\xymatrix { {{M}/{Z(M)}} \times {{M}/{Z(M)}} \ar[rr]^-{c_{M}} \ar[d]_{{\eta}\times{\eta}}  & & [M,M] \ar[d]^{\xi} \\ {N}/{Z(N)}\times{N}/{Z(N)} \ar[rr]_-{c_{N}}  & & [N,N] }
\end{equation*}

\noindent where $c_{M},c_{N}$ are commutator maps of groups. The pair $(\eta
,\xi )$ is called an isoclinism from $M$ to $N,$ and denoted by $(\eta ,\xi
):M\sim N.$

\begin{remark}
As expected, isoclinism is an equivalence relation.
\end{remark}

\begin{examples}
\ \ \
\end{examples}

(1) All Abelian groups are isoclinic to each other.

(2) The dihedral, quasidihedral and quaternion groups of order $2^{n}$ are
isoclinic, for $n\geq 3.$

(3) Every group is isoclinic to a stem group. (See \cite{PH}, for details.)%
\newline

Now we are going to define the notion of isoclinic crossed modules.

{\noindent \textbf{Notation}} In the sequel of the paper, for a given
crossed module $G:G_{1}\overset{d_{G}}{\longrightarrow }G_{0},$ we denote $%
G/Z(G)$ by $\overline{G_{1}}\overset{\overline{d_{G}}}{\longrightarrow }%
\overline{G_{0}}$ where $\overline{G_{1}}=G_{1}/G_{1}^{G_{0}}$ and $%
\overline{G_{0}}=G_{0}/(St_{G_{0}}(G_{1})\cap Z(G_{0}))$, for shortness.

\begin{definition}
\label{04} The crossed modules $G:G_{1}\overset{d_{G}}{\longrightarrow }%
G_{0} $ and $H:H_{1}\overset{d_{H}}{\longrightarrow }H_{0}$ are \textit{%
isoclinic} if there exist isomorphisms
\end{definition}

\begin{equation*}
(\eta _{1},\eta _{0}):(\overline{G_{1}}\overset{\overline{d_{G}}}{%
\longrightarrow }\overline{G_{0}})\longrightarrow (\overline{H_{1}}\overset{%
\overline{d_{H}}}{\longrightarrow }\overline{H_{0}})
\end{equation*}%
and%
\begin{equation*}
(\xi _{1},\xi _{0}):(D_{G_{0}}(G_{1})\overset{d_{G}|}{\longrightarrow }%
[G_{0},G_{0}])\longrightarrow (D_{H_{0}}(H_{1})\overset{d_{H}|}{%
\longrightarrow }[H_{0},H_{0}])
\end{equation*}%
such that the diagrams%
\begin{equation}
\xymatrix { {\overline{G_{1}}} \times {\overline{G_{0}}} \ar[rr]^ -{c_1} \ar[d]_{{{\eta}_1} \times {{\eta}_0}}  & & D_{G_{0}}{(G_{1})} \ar[d]^{{\xi}_1} \\ {\overline{H_{1}}} \times {\overline{H_{0}}}   \ar[rr]_ -{{c_1}'}  & & D_{H_{0}}(H_{1}) }
\end{equation}%
and%
\begin{equation}
\xymatrix { {\overline{G_{0}}} \times {\overline{G_{0}}} \ar[rr]^ -{c_0} \ar[d]_{{{\eta}_0} \times {{\eta}_0}}  & & [G_0,G_0] \ar[d]^{{\xi}_0} \\ {\overline{H_{0}}} \times {\overline{H_{0}}}   \ar[rr]_ -{{c_0}'}  & & [H_0,H_0] }
\end{equation}%
are commutative where $c_{1},c_{1}^{\prime }$ defined by $%
c_{1}(g_{1}G_{1}^{G_{0}},g_{0}(St_{G_{0}}(G_{1})\cap Z(G_{0})))=$ $%
^{g_{0}}g_{1}g_{1}^{-1}$,\newline
$c_{1}^{\prime }(h_{1}H_{1}^{H_{0}},h_{0}(St_{H_{0}}(H_{1})\cap Z(H_{0})))=$
$^{h_{0}}h_{1}h_{1}^{-1},$ for all $g_{1}\in G_{1},$ $g_{0}\in G_{0},$ $%
h_{0}\in H_{0},$ $h_{1}\in H_{1}$

\noindent and $c_{0},c_{0}^{\prime }$ defined by $%
c_{0}(g_{0}(St_{G_{0}}(G_{1})\cap Z(G_{0})),$ $g_{0}^{^{\prime
}}(St_{G_{0}}(G_{1})\cap Z(G_{0})))=[g_{0},g_{0}^{\prime }],$ $c_{0}^{\prime
}(h_{0}(St_{H_{0}}(H_{1})\cap Z(H_{0})),h_{0}^{\prime
}(St_{G_{0}}(G_{1})\cap Z(G_{0})))=[h_{0},h_{0}^{\prime }],$ for all $%
g_{0},g_{0}^{\prime }\in G_{0}$ and $h_{0},h_{0}^{\prime }\in H_{0}.$ (The
well definition of the maps $c_{1},c_{0},c_{1}^{\prime }$ and $c_{0}^{\prime
}$ are given in Appendix A.)

The pair $((\eta _{1},\eta _{0}),(\xi _{1},\xi _{0}))$ is called an \textit{%
isoclinism} from $G$ to $H$ and this situation is denoted by $((\eta
_{1},\eta _{0}),(\xi _{1},\xi _{0})):G\sim H.$

\begin{remark}
If the crossed modules $G$ and $H$ are simply connected or finite, then the
commutativity of diagrams $(1)$ with $(2)$ in Definition \ref{04} are
equivalent to the commutativity of following diagram:%
\begin{equation*}
\xymatrix { {G/{Z(G)}} \times {G/{Z(G)}} \ar[rr]^-{} \ar[d]_{({\eta_1}, {\eta_0})\times({\eta_1}, {\eta_0})}  & & [G,G] \ar[d]^{({\xi_1},{\xi_0})} \\ {H/{Z(H)}} \times {H/{Z(H)}} \ar[rr]_-{}  & & [H,H] }
\end{equation*}
\end{remark}

\begin{examples}
\ \ \
\end{examples}

(1) All Abelian crossed modules are isoclinic.

(2) Let $M$ and $N$ be isoclinic groups. Then $M\overset{id}{\longrightarrow
}M$ is isoclinic to $N\overset{id}{\longrightarrow }N.$

(3) Let $M$ be a group and let $N$ be a normal subgroup of $M$ with $%
NZ(M)=M. $ Then {\xymatrix {N \ar@{^{(}->}[r]^{inc.} &M}} is isoclinic to $M%
\overset{id}{\longrightarrow }M.$

(4) Some particular examples can be found in Section $5$.

\begin{proposition}
Isoclinism is an equivalence relation.
\end{proposition}

\begin{proof}
One can easily check by a direct calculation.
\end{proof}

In \cite{PH}, it is proved that every group is isoclinic to a stem
group. This property depents on the construction of Schur multiplicator and
stem extensions of a group as given in \cite{TB}. The same definitions and
constructions were given for crossed modules in \cite{VC}.

\begin{proposition}
Every crossed module is isoclinic to a stem crossed module.
\end{proposition}

\begin{proof}
It can be proved by using the related constructions and definitions given in
\cite{TB,VC}.
\end{proof}

\begin{proposition}
\label{05} Let $G:G_{1}\overset{d}{\longrightarrow }G_{0}$ be a crossed
module and $H:H_{1}\overset{d|}{\longrightarrow }H_{0}$ be its subcrossed
module. If $G=HZ(G),$ i.e $G_{1}=H_{1}G_{1}^{G_{0}}$ and $%
G_{0}=H_{0}(St_{G_{0}}(G_{1})\cap Z(G_{0})),$ then $G$ is isoclinic to $H.$
\end{proposition}

\begin{proof}
First, we will show that $H_{1}^{H_{0}}=H_{1}\cap G_{1}^{G_{0}}$ and $%
St_{H_{0}}(H_{1})\cap Z(H_{0})=H_{0}\cap (St_{G_{0}}(G_{1})\cap Z(G_{0})).$

Let $h_{1}\in H_{1}^{H_{0}}.$ For any $g_{0}\in G_{0},$ since $%
G_{0}=H_{0}(St_{G_{0}}(G_{1})\cap Z(G_{0}))$ there exist $a_{0}\in
St_{G_{0}}(G_{1})\cap Z(G_{0})$ and $h_{0}^{\prime }\in H_{0}$ such that $%
g_{0}=h_{0}^{\prime }a_{0}.$ We have $^{g_{0}}h_{1}=$ $^{(h_{0}^{\prime
}a_{0})}h_{1}=$ $^{h_{0}^{\prime }}(^{a_{0}}h_{1})=$ $^{h_{0}^{\prime
}}h_{1}=h_{1},$ so $h_{1}\in H_{1}\cap G_{1}^{G_{0}}.$ Conversely, for any $%
h_{1}\in H_{1}\cap G_{1}^{G_{0}},$ we have $h_{1}\in H_{1}^{H_{0}}.$ So, $%
H_{1}^{H_{0}}=H_{1}\cap G_{1}^{G_{0}}.$

Let $h_{0}\in St_{H_{0}}(H_{1})\cap Z(H_{0}).$ For any $g_{1}\in G_{1},$
there exist $k_{1}\in H_{1}$ and $a_{1}\in G_{1}^{G_{0}}$ such that $%
g_{1}=k_{1}a_{1}.$ Then%
\begin{equation*}
^{h_{0}}g_{1}=\text{ }^{h_{0}}(k_{1}a_{1})=\text{ }^{h_{0}}k_{1}\text{ }%
^{h_{0}}a_{1}=k_{1}a_{1},
\end{equation*}%
which means that $h_{0}\in St_{G_{0}}(G_{1}).$ On the other hand, it is
clear that $h_{0}\in Z(G_{0}).$ Then, we obtain $h_{0}\in H_{0}\cap
(St_{G_{0}}(G_{1})\cap Z(G_{0})).$ By a direct calculation we get $%
St_{H_{0}}(H_{1})\cap Z(H_{0})=H_{0}\cap (St_{G_{0}}(G_{1})\cap Z(G_{0})).$

By the third isomorphism theorem for crossed modules, we have%
\begin{equation*}
\begin{array}{lll}
\dfrac{H}{Z(H)} & = & \dfrac{(H_{1},H_{0},d|)}{%
(H_{1}^{H_{0}},St_{H_{0}}(H_{1})\cap Z(H_{0}),d|)} \\
& = & \dfrac{(H_{1},H_{0},d|)}{(H_{1}\cap G_{1}^{G_{0}},H_{0}\cap
(St_{G_{0}}(G_{1})\cap Z(G_{0})),d|)} \\
& = & \dfrac{(H_{1},H_{0},d|)}{(G_{1}^{G_{0}},St_{G_{0}}(G_{1})\cap
Z(G_{0}),d|)\cap (H_{1},H_{0},d|)} \\
& \cong  & \dfrac{(H_{1},H_{0},d|)(G_{1}^{G_{0}},St_{G_{0}}(G_{1})\cap
Z(G_{0}),d|)}{(G_{1}^{G_{0}},St_{G_{0}}(G_{1})\cap Z(G_{0}),d|)} \\
& = & \dfrac{HZ(G)}{Z(G)} \\
& = & \dfrac{G}{Z(G)},%
\end{array}%
\end{equation*}%
as required.

Let $^{g_{0}}g_{1}g_{1}^{-1}\in D_{G_{0}}(G_{1}),$ then there exist $%
h_{1}\in H_{1},$ $a_{1}\in G_{1}^{G_{0}},$ $h_{0}\in H_{0},$ $a_{0}\in
(St_{G_{0}}(G_{1})\cap Z(G_{0}))$ such that $g_{1}=h_{1}a_{1}$ and $%
g_{0}=h_{0}a_{0}.$ Since%
\begin{equation*}
\begin{array}{lll}
^{g_{0}}g_{1}g_{1}^{-1} & = & ^{(h_{0}a_{0})}(h_{1}a_{1})(h_{1}a_{1})^{-1}
\\
& = & ^{h_{0}a_{0}}(h_{1})^{h_{0}a_{0}}(a_{1})(h_{1}a_{1})^{-1} \\
& = & ^{h_{0}}(^{a_{0}}h_{1})^{h_{0}}(^{a_{0}}a_{1})a_{1}^{-1}h_{1}^{-1} \\
& = & (^{h_{0}}h_{1})(^{h_{0}}a_{1})a_{1}^{-1}h_{1}^{-1} \\
& = & (^{h_{0}}h_{1})a_{1}a_{1}^{-1}h_{1}^{-1} \\
& = & ^{h_{0}}h_{1}h_{1}^{-1},%
\end{array}%
\end{equation*}%
we have $^{g_{0}}g_{1}g_{1}^{-1}\in D_{H_{0}}(H_{1}).$ On the other hand,
for any $[g_{0},g_{0}^{\prime }]\in \lbrack G_{0},G_{0}]$ there exist $%
h_{0},h_{0}^{\prime }\in H_{0},$ $a_{0},a_{0}^{\prime }\in
(St_{G_{0}}(G_{1})\cap Z(G_{0}))$ such that $g_{0}=h_{0}a_{0},$ $%
g_{0}^{\prime }=h_{0}^{\prime }a_{0}^{\prime },$ from which we get $%
[g_{0},g_{0}^{\prime }]=[h_{0}a_{0},h_{0}^{\prime }a_{0}^{\prime
}]=[h_{0},h_{0}^{\prime }].$

Finally, we have that the crossed modules $G$ and $H$ are isoclinic where
the isomorphisms $(\eta _{1},\eta _{0})$ and $(\xi _{1},\xi _{0})$ are
defined by $(inc.,inc.)$, $(id_{G_{1}},id_{G_{0}})$, respectively.
\end{proof}

\begin{remark}
When $H:H_{1}\overset{d|}{\longrightarrow }H_{0}$ is finite crossed module
then the converse of Proposition \ref{05} is true.
\end{remark}

\begin{proposition}
Let $G:G_{1}\overset{d_{G}}{\longrightarrow }G_{0}$ and $H:H_{1}\overset{%
d_{H}}{\longrightarrow }H_{0}$ be isoclinic crossed mo\-dules.\newline
(i) If $G$ and $H$ are aspherical, then $G_{0}$ and $H_{0}$ are isoclinic
groups$.$\newline
(ii) If $G$ and $H$ are simply connected, then $G_{1}$ and $H_{1}$ are
isoclinic groups$.$
\end{proposition}

\begin{proof}
Let $G:G_{1}\overset{d_{G}}{\longrightarrow }G_{0}$ and $H:H_{1}\overset{%
d_{H}}{\longrightarrow }H_{0}$ be isoclinic crossed modules. Then we have
the crossed module isomorphisms%
\begin{eqnarray*}
(\eta _{1},\eta _{0}) &:&(\overline{G_{1}}\overset{\overline{d_{G}}}{%
\longrightarrow }\overline{G_{0}})\longrightarrow (\overline{H_{1}}\overset{%
\overline{d_{H}}}{\longrightarrow }\overline{H_{0}}) \\
(\xi _{1},\xi _{0}) &:&(D_{G_{0}}(G_{1})\overset{d_{G}|}{\longrightarrow }%
[G_{0},G_{0}])\longrightarrow (D_{H_{0}}(H_{1})\overset{d_{H}|}{%
\longrightarrow }[H_{0},H_{0}])
\end{eqnarray*}%
which makes diagrams (1) and (2) commutative.

(i) From asphericallity of the crossed modules, we have $Z(G_{0})\subseteq
St_{G_{0}}(G_{1}),$ $Z(H_{0})\subseteq St_{H_{0}}(H_{1}).$ Consequently, $%
\eta _{0}$ is an isomorphism between $G_{0}/Z(G_{0})$ and $H_{0}/Z(H_{0}).$
So the isomorphisms $\eta _{0}$ and $\xi _{0}$ give rise to an isoclinism
from $G_{0}$ to $H_{0}.$

(ii) Since $G$ and $H$ are simply connected crossed modules, we have $%
G_{1}^{G_{0}}=Z(G_{1}),$ $H_{1}^{H_{0}}=Z(H_{1}),$ $%
D_{G_{0}}(G_{1})=[G_{1},G_{1}]$ and $D_{H_{0}}(H_{1})=[H_{1},H_{1}].$ So we
have the isomorphisms $\eta _{1}:G_{1}/Z(G_{1})\longrightarrow
H_{1}/Z(H_{1}),$ $\xi _{1}:[G_{1},G_{1}]\longrightarrow \lbrack H_{1},H_{1}]$
which make $G_{1}$ and $H_{1}$ isoclinic.
\end{proof}

\begin{proposition}
Let $G$ and $H$ be isoclinic finite crossed modules. Then $G_{1}$ and $G_{0}$
are isoclinic to $H_{1}$ and $H_{0},$ respectively.
\end{proposition}

\begin{proof}
Let $G:G_{1}\overset{d_{G}}{\longrightarrow }G_{0}$ and $H:H_{1}\overset{%
d_{H}}{\longrightarrow }H_{0}$ be isoclinic crossed module. Then we have the
crossed module isomorphisms%
\begin{eqnarray*}
(\eta _{1},\eta _{0}) &:&(\overline{G_{1}}\overset{\overline{d_{G}}}{%
\longrightarrow }\overline{G_{0}})\longrightarrow (\overline{H_{1}}\overset{%
\overline{d_{H}}}{\longrightarrow }\overline{H_{0}}) \\
(\xi _{1},\xi _{0}) &:&(D_{G_{0}}(G_{1})\overset{d_{G}|}{\longrightarrow }%
[G_{0},G_{0}])\longrightarrow (D_{H_{0}}(H_{1})\overset{d_{H}|}{%
\longrightarrow }[H_{0},H_{0}])
\end{eqnarray*}%
which makes diagrams (1) and (2) commutative. The isomorphism $\xi
_{1}:D_{G_{0}}G_{1}\longrightarrow D_{H_{0}}H_{1}$ gives rise to the
restriction $\xi _{1}|:[G_{1},G_{1}]\longrightarrow \lbrack H_{1},H_{1}]$
which is also an isomorphism by the finiteness of $G_{1}$ and $H_{1}$.
Similarly, we have the isomorphisms $\eta _{1}^{\prime
}:G_{1}/Z(G_{1})\longrightarrow H_{1}/Z(H_{1}),\eta _{1}^{\prime
}(g_{1}Z(G_{1}))=h_{1}Z(H_{1})$, $\eta _{0}^{\prime
}:G_{0}/Z(G_{0})\longrightarrow H_{0}/Z(H_{0}),\eta _{0}^{\prime
}(g_{0}Z(G_{0}))=h_{0}Z(H_{0})$, and $\xi _{0}$ which makes $G_{1},$ $G_{0}$
isoclinic to $H_{1}$, $H_{0},$ respectively.
\end{proof}

\begin{remark}
In general, the finiteness of a crossed module $G:G_{1}\overset{d}{%
\longrightarrow }G_{0}$ does not give the equation $%
[G_{1},G_{1}]=D_{G_{0}}(G_{1}).$ For the crossed module $C_{8}\overset{d}{%
\longrightarrow }C_{2},$ we have $[C_{8},C_{8}]=\{e\}$ and $%
D_{C_{2}}(C_{8})=C_{4}$. So $[C_{8},C_{8}]\neq D_{C_{2}}(C_{8}).$
\end{remark}

As indicated in \cite{MH}, the isoclinism of two groups doesn't give rise to
the isomorphism of their automorphism groups.

\begin{example}
\label{enver} Despite the fact that $C_{32}$ isoclinic to $Kl(4)$, $%
Aut(C_{32})$ isn't isomorphic to $Aut(Kl_{4})$. Since, $\left\vert
Aut(C_{32})\right\vert =16\neq 6=\left\vert Aut(Kl(4))\right\vert .$ But, in
\cite{YA}, it is shown that, class preserving \ automorphism groups of
isoclinic groups are isomorphic.
\end{example}

In the crossed module case, we obtain the same results. For this, first we
introduce the class preserving actor of a crossed module as follows;

\begin{proposition}
\label{06a} Let $G:G_{1}\longrightarrow G_{0}$ be a crossed module, $D_{%
\mathcal{C}}(G_{0},G_{1})=\{\delta \in D(G_{0},G_{1})$ $|$ there exists $%
g_{1}\in G_{1}$ such that $\delta (g_{0})=g_{1}$ $^{g_{0}}g_{1}^{-1},$ for
all $g_{0}\in G_{0}\}$ and $Aut_{_{\mathcal{C}}}(G)=\{(\alpha ,\beta )\in
Aut(G)$ $|$ there exists $g_{0}\in G_{0}$ such that $\alpha (g_{1})=$ $%
^{g_{0}}g_{1},$ $\beta (g_{0}^{\prime })=g_{0}g_{0}^{\prime }g_{0}^{-1},$
for all $g_{1}\in G_{1},$ $g_{0}^{\prime }\in G_{0}\}$. Then, we have the
following:\newline
(a) $D_{\mathcal{C}}(G_{0},G_{1})$ is a subgroup of $D(G_{0},G_{1}).$\newline
(b) $Aut_{_{\mathcal{C}}}(G)$ is a subgroup of $Aut(G).$
\end{proposition}

\begin{proof}
The proof is given in Appendix A.
\end{proof}

\begin{proposition}
\label{06} $Act_{\mathcal{C}}(G):D_{\mathcal{C}}(G_{0},G_{1})\overset{\Delta
_{_{\mathcal{C}}}}{\longrightarrow }Aut_{_{\mathcal{C}}}(G)$ is a crossed
module with the action induced from the action of $Aut(G)$ over $%
D(G_{0},G_{1})$ such that%
\begin{equation*}
\begin{array}{lll}
Aut_{_{\mathcal{C}}}(G)\times D_{\mathcal{C}}(G_{0},G_{1}) & \longrightarrow
& D_{\mathcal{C}}(G_{0},G_{1}) \\
((\alpha ,\beta ),\delta ) & \longmapsto & ^{(\alpha ,\beta )}\delta%
\end{array}%
\end{equation*}%
$^{(\alpha ,\beta )}\delta (h_{0})=\alpha (g_{1})^{h_{0}}\alpha (g_{1})^{-1}$%
, for all $h_{0}\in G_{0}$ and $g_{1}\in G_{1}.$
\end{proposition}

\begin{proof}
It can be shown by a direct calculation.
\end{proof}

\begin{definition}
Let $G:G_{1}\overset{d_{G}}{\longrightarrow }G_{0}$ be a crossed module. The
crossed module%
\begin{equation*}
Act_{\mathcal{C}}(G):D_{\mathcal{C}}(G_{0},G_{1})\overset{\Delta _{_{%
\mathcal{C}}}}{\longrightarrow }Aut_{_{\mathcal{C}}}(G)
\end{equation*}%
define in Proposition \ref{06} will be called \textit{class preserving actor}
of $G$ and will be denoted by $Act_{\mathcal{C}}(G).$
\end{definition}

\begin{theorem}
Let $G:G_{1}\overset{d_{G}}{\longrightarrow }G_{0}$ and $H:H_{1}\overset{%
d_{H}}{\longrightarrow }H_{0}$ be isoclinic crossed modules. Then, we have $%
D_{\mathcal{C}}(G_{0},G_{1})\cong D_{\mathcal{C}}(H_{0},H_{1}).$
\end{theorem}

\begin{proof}
Suppose that $G$ and $H$ are isoclinic crossed modules. Then we have the
isomorphisms%
\begin{eqnarray*}
(\eta _{1},\eta _{0}) &:&(\overline{G_{1}}\overset{\overline{d_{G}}}{%
\longrightarrow }\overline{G_{0}})\longrightarrow (\overline{H_{1}}\overset{%
\overline{d_{H}}}{\longrightarrow }\overline{H_{0}}) \\
(\xi _{1},\xi _{0}) &:&(D_{G_{0}}(G_{1})\overset{d_{G}|}{\longrightarrow }%
[G_{0},G_{0}])\longrightarrow (D_{H_{0}}(H_{1})\overset{d_{H}|}{%
\longrightarrow }[H_{0},H_{0}])
\end{eqnarray*}%
which makes the diagrams (1) and (2) commutative. Let $\delta _{G}\in D_{%
\mathcal{C}}(G_{0},G_{1})$, $h_{0}\in H_{0}-(St_{H_{0}}(H_{1})\cap Z(H_{0}))$
and $\overline{h_{0}}=h_{0}St_{H_{0}}(H_{1})\cap Z(H_{0})\in \overline{H_{0}}%
.$ Define $\overline{g_{0}}=g_{0}St_{G_{0}}(G_{1})\cap Z(G_{0})=\eta
_{0}^{-1}(\overline{h_{0}})\in \overline{G_{0}}.$ Since $g_{0}\in G_{0},$
there exists an element $a_{1}\in G_{1}$ such that $\delta _{G}(g_{0})=a_{1}$
$^{g_{0}}a_{1}^{-1}.$

Let $\eta _{1}(\overline{a_{1}})=\overline{a_{1}^{\prime }}.$ Now we define
a map $\delta _{H}:H_{0}\longrightarrow H_{1}$ by%
\begin{equation*}
\delta _{H}(h_{o})=\left\{
\begin{array}{lllc}
a_{1}^{\prime }\text{ }^{h_{0}}(a_{1}^{\prime })^{-1} &  & h_{0}\in
H_{0}-(St_{H_{0}}(H_{1})\cap Z(H_{0})) &  \\
e_{H_{1}} &  & h_{0}\in St_{H_{0}}(H_{1})\cap Z(H_{0}). &
\end{array}%
\right.
\end{equation*}%
We will give the proof step-by-step.

\textbf{Step 1 }$\delta _{H}$ is well defined.

Proof: Let $h_{0},h_{0}^{\prime }\in H_{0}-(St_{H_{0}}(H_{1})\cap Z(H_{0})).$
Then $\overline{h_{0}},$ $\overline{h_{0}^{\prime }}\in \overline{H_{0}}$
and $\overline{g_{0}}=\eta _{0}^{-1}(\overline{h_{0}})\in \overline{G_{0}},$
$\overline{g_{0}^{\prime }}=\eta _{0}^{-1}(\overline{h_{0}^{\prime }})\in
\overline{G_{0}}.$ So $g_{0},$ $g_{0}^{\prime }\in G_{0}$ and there exist
elements $a_{1},$ $b_{1}\in G_{1}$ such that $\delta _{G}(g_{0})=a_{1}$ $%
^{g_{0}}a_{1}^{-1}$ and $\delta _{G}(g_{0}^{\prime })=b_{1}$ $%
^{g_{0}^{\prime }}b_{1}^{-1}.$ Now we will show that if $h_{0}=h_{0}^{\prime
},$ then $a_{1}^{\prime }$ $^{h_{0}}(a_{1}^{\prime })^{-1}=b_{1}^{\prime }$ $%
^{h_{0}^{\prime }}(b_{1}^{\prime })^{-1}$ where $\overline{a_{1}^{\prime }}%
=\eta _{1}(\overline{a_{1}}),$ $\overline{b_{1}^{\prime }}=\eta _{1}(%
\overline{b_{1}})$.

Let $h_{0}=h_{0}^{\prime }.$ Then $\overline{g_{0}}=\overline{g_{0}^{\prime }%
}$ which means $(g_{0}^{\prime })^{-1}g_{0}\in St_{G_{0}}(G_{1})\cap
Z(G_{0}).$ Say, $(g_{0}^{\prime })^{-1}g_{0}=g$ where $g\in
St_{G_{0}}(G_{1})\cap Z(G_{0}).$ We have%
\begin{eqnarray*}
\delta _{G}(g_{0}) &=&\delta _{G}(g_{0}^{\prime }g) \\
&=&\delta _{G}(g_{0}^{\prime })^{g_{0}^{\prime }}\delta _{G}(g)\text{ \ }%
(\because \delta _{G}\in D(G_{0},G_{1})) \\
&=&\delta _{G}(g_{0}^{\prime })^{g_{0}^{\prime }}e_{G_{1}}\text{ \ \ \ \ \ }%
(\because g\in St_{G_{0}}(G_{1})\cap Z(G_{0})) \\
&=&\delta _{G}(g_{0}^{\prime }).
\end{eqnarray*}%
So, $a_{1}$ $^{g_{0}}a_{1}^{-1}=b_{1}$ $^{g_{0}^{\prime }}b_{1}^{-1}.$ By
applying $\xi _{1},$ we get $a_{1}^{\prime }$ $^{h_{0}}(a_{1}^{\prime
})^{-1}=b_{1}^{\prime }$ $^{h_{0}^{\prime }}(b_{1}^{\prime })^{-1}$ which
shows the well definition of $\delta _{H}.$

\textbf{Step 2 }$\delta _{H}\in D_{\mathcal{C}}(H_{0},H_{1}).$

Proof: Let $h_{0},h_{0}^{\prime }\in H_{0}.$ We must check that $\delta
_{H}(h_{0}h_{0}^{\prime })=\delta _{H}(h_{0})(^{h_{0}}\delta
_{H}(h_{o}^{\prime })).$

(i) If $h_{0},h_{0}^{\prime }\in St_{H_{0}}(H_{1})\cap Z(H_{0}),$ then $%
h_{0}h_{0}^{\prime }\in St_{H_{0}}(H_{1})\cap Z(H_{0}).$ So, $\delta
_{H}(h_{0}h_{0}^{\prime })=e_{H_{1}}$ and $\delta _{H}(h_{0})(^{h_{0}}\delta
_{H}(h_{0}^{\prime }))=e_{H_{1}}$ $^{h_{0}}e_{H_{1}}=e_{H_{1}}.$ That is, $%
\delta _{H}(h_{0}h_{0}^{\prime })=\delta _{H}(h_{0})(^{h_{0}}\delta
_{H}(h_{0}^{\prime })).$

(ii) If $h_{0}\in H_{0}-(St_{H_{0}}(H_{1})\cap Z(H_{0}))$ and $h_{0}^{\prime
}\in St_{H_{0}}(H_{1})\cap Z(H_{0}),$ then $h_{0}h_{0}^{\prime }\in
H_{0}-(St_{H_{0}}(H_{1})\cap Z(H_{0})).$ Since $\delta
_{H}(h_{0})=a_{1}^{\prime }$ $^{h_{0}}(a_{1}^{\prime })^{-1}$ and $\delta
_{H}(h_{0}^{\prime })=e_{H_{1}},$ we have%
\begin{eqnarray*}
\delta _{H}(h_{0}h_{0}^{\prime }) &=&a_{1}^{\prime }\text{ }%
^{h_{0}h_{0}^{\prime }}(a_{1}^{\prime })^{-1} \\
&=&a_{1}^{\prime }\text{ }^{h_{0}h_{0}^{\prime }}(^{(h_{0}^{\prime
})^{-1}}(a_{1}^{\prime })^{-1}) \\
&=&a_{1}^{\prime }\text{ }^{(h_{0}h_{0}^{\prime }(h_{0}^{\prime
})^{-1})}(a_{1}^{\prime })^{-1} \\
&=&a_{1}^{\prime }\text{ }^{h_{0}}(a_{1}^{\prime })^{-1}e_{H_{1}} \\
&=&(a_{1}^{\prime }\text{ }^{h_{0}}(a_{1}^{\prime })^{-1})^{h_{0}}e_{H_{1}}
\\
&=&\delta _{H}(h_{0})(^{h_{0}}\delta _{H}(h_{0}^{\prime })).
\end{eqnarray*}

(iii) If $h_{0},h_{0}^{\prime }\in H_{0}-(St_{H_{0}}(H_{1})\cap Z(H_{0})),$
then $h_{0}h_{0}^{\prime }\in H_{0}-(St_{H_{0}}(H_{1})\cap Z(H_{0}))$ and $%
\eta _{0}^{-1}(\overline{h_{0}h_{0}^{\prime }})=\eta _{0}^{-1}(\overline{%
h_{0}})\eta _{0}^{-1}(\overline{h_{0}^{\prime }})=\overline{g_{0}}\overline{%
g_{0}^{\prime }}=\overline{g_{0}g_{0}^{\prime }},$ where $g_{0},$ $%
g_{0}^{\prime }\in G_{0}.$ Let $\delta _{G}(g_{0}g_{0}^{\prime })=a_{1}b_{1}$
$^{g_{0}g_{0}^{\prime }}(a_{1}b_{1})^{-1},$ $\delta _{G}(g_{0})=a_{1}$ $%
^{g_{0}}a_{1}^{-1}$ and $\delta _{G}(g_{0}^{\prime })=b_{1}$ $%
^{g_{0}^{\prime }}b_{1}^{-1}.$ Since $\delta _{G}(g_{0}g_{0}^{\prime
})=\delta _{G}(g_{0})^{g_{0}}\delta _{G}(g_{0}^{\prime }),$ we get%
\begin{eqnarray*}
a_{1}b_{1}\text{ }^{g_{0}g_{0}^{\prime }}(a_{1}b_{1})^{-1} &=&(a_{1}\text{ }%
^{g_{0}}a_{1}^{-1})\text{ }^{g_{0}}(b_{1}\text{ }^{g_{0}^{\prime
}}b_{1}^{-1}) \\
&=&(a_{1}\text{ }^{g_{0}}a_{1}^{-1})(^{g_{0}}b_{1}\text{ }%
^{g_{0}g_{0}^{\prime }}b_{1}^{-1}) \\
&=&(a_{1}\text{ }^{g_{0}}a_{1}^{-1})(^{g_{0}}b_{1}\text{ }(b_{1}^{-1}b_{1})%
\text{\ }^{g_{0}g_{0}^{\prime }}b_{1}^{-1}) \\
&=&(a_{1}\text{ }^{g_{0}}a_{1}^{-1})(^{g_{0}}b_{1}b_{1}^{-1})(b_{1}\text{ }%
^{g_{0}g_{0}^{\prime }}b_{1}^{-1}).
\end{eqnarray*}%
By applying $\xi _{1},$ we get%
\begin{eqnarray*}
a_{1}^{\prime }b_{1}^{\prime }\text{ }^{h_{0}h_{0}^{\prime }}(a_{1}^{\prime
}b_{1}^{\prime })^{-1} &=&(a_{1}^{\prime }\text{ }^{h_{0}}(a_{1}^{\prime
})^{-1})(^{h_{0}}b_{1}^{\prime }(b_{1}^{\prime })^{-1})(b_{1}^{\prime }\text{
}^{h_{0}h_{0}^{\prime }}(b_{1}^{\prime })^{-1}) \\
&=&(a_{1}^{\prime }\text{ }^{h_{0}}(a_{1}^{\prime
})^{-1})(^{h_{0}}b_{1}^{\prime }\text{ }^{h_{0}h_{0}^{\prime
}}(b_{1}^{\prime })^{-1}) \\
&=&(a_{1}^{\prime }\text{ }^{h_{0}}(a_{1}^{\prime
})^{-1})^{h_{0}}(b_{1}^{\prime }\text{ }^{h_{0}^{\prime }}(b_{1}^{\prime
})^{-1}).
\end{eqnarray*}%
That is, $\delta _{H}(h_{0}h_{0}^{\prime })=\delta _{H}(h_{0})^{h_{0}}\delta
_{H}(h_{o}^{\prime }).$\newline
From definition of $\delta _{H}$, we obtain $\delta _{H}\in D_{\mathcal{C}%
}(H_{0},H_{1}).$

\textbf{Step 3 }The map%
\begin{equation*}
\begin{array}{clll}
\phi : & D_{\mathcal{C}}(G_{0},G_{1}) & \longrightarrow & D_{\mathcal{C}%
}(H_{0},H_{1}) \\
& \delta _{G} & \longmapsto & \delta _{H}%
\end{array}%
\end{equation*}%
is an isomorphism.

Proof: Let $\delta _{G},$ $\delta _{G}^{\prime }\in D_{\mathcal{C}%
}(G_{0},G_{1}),$ $h_{0}\in H_{0}$ and $\overline{g_{0}}=\eta _{0}^{-1}(%
\overline{h_{0}}).$ Since $\delta _{G}\delta _{G}^{\prime }\in D_{\mathcal{C}%
}(G_{0},G_{1}),$ there exists $a\in G_{1}$ such that $(\delta _{G}\delta
_{G}^{\prime })(g_{0})=a$ $^{g_{0}}a^{-1}.$ Also there exist $a_{1},b_{1}\in
G_{1}$ such that $\delta _{G}(g_{0})=a_{1}$ $^{g_{0}}a_{1}^{-1},$ $\delta
_{G}^{\prime }(g_{0})=b_{1}$ $^{g_{0}}b_{1}^{-1}.$ Since $(\delta _{G}\delta
_{G}^{\prime })(g_{0})=a_{1}b_{1}$ $^{g_{0}}(a_{1}b_{1})^{-1}$, we get $a$ $%
^{g_{0}}a^{-1}=a_{1}b_{1}$ $^{g_{0}}(a_{1}b_{1})^{-1}.$

By applying $\xi _{1},$ we get%
\begin{equation*}
a^{\prime }\text{ }^{h_{0}}(a^{\prime })^{-1}=a_{1}^{\prime }b_{1}^{\prime }%
\text{ }^{h_{0}}(a_{1}^{\prime }b_{1}^{\prime })^{-1},
\end{equation*}%
since $\eta _{1}(\overline{a_{1}b_{1}})=\eta _{1}(\overline{a_{1}})\eta _{1}(%
\overline{b_{1}})=\overline{a_{1}^{\prime }}\overline{b_{1}^{\prime }}=%
\overline{a_{1}^{\prime }b_{1}^{\prime }}.$ Finally, since%
\begin{equation*}
\phi _{\delta _{G}\delta _{G}^{\prime }}(h_{0})=a^{\prime }\text{ }%
^{h_{0}}(a^{\prime })^{-1}
\end{equation*}%
and%
\begin{equation*}
\phi _{\delta _{G}}\phi _{\delta _{G}^{\prime }}(h_{0})=a_{1}^{\prime
}b_{1}^{\prime }\text{ }^{h_{0}}(a_{1}^{\prime }b_{1}^{\prime })^{-1},\text{
for all }h_{0}\in H_{0},
\end{equation*}%
we have $\phi _{\delta _{G}\delta _{G}^{\prime }}=\phi _{\delta _{G}}\phi
_{\delta _{G}^{\prime }},$ i.e $\phi $ is a homomorphism.

Similarly, for each $\delta _{H}\in D_{\mathcal{C}}(H_{0},H_{1}),$ we can
define $\delta _{G}\in D_{\mathcal{C}}(G_{0},G_{1})$ and the homomorphism $%
\varphi :D_{\mathcal{C}}(H_{0},H_{1})\rightarrow D_{\mathcal{C}%
}(G_{0},G_{1}),$ $\varphi (\delta _{H})=\delta _{G}$ as follows;

Let $\delta _{H}\in D_{\mathcal{C}}(H_{0},H_{1})$ and $g_{0}\in
G_{0}-(St_{G_{0}}(G_{1})\cap Z(G_{0})).$ Define $\overline{h_{0}}%
=h_{0}(St_{H_{0}}(H_{1})\cap Z(H_{0}))=\eta _{0}(\overline{g_{0}})\in
\overline{H_{0}}.$ So there exists an element $a_{1}^{\prime }\in H_{1}$
such that $\delta _{H}(h_{0})=a_{1}^{\prime }$ $^{h_{0}}(a_{1}^{\prime
})^{-1}.$

Let $\eta _{1}^{-1}(\overline{a_{1}^{\prime }})=\overline{a_{1}}.$ Now we
define the map $\delta _{G}:G_{0}\longrightarrow G_{1}$ by%
\begin{equation*}
\delta _{G}(g_{o})=\left\{
\begin{array}{llll}
a_{1}\text{ }^{g_{0}}(a_{1})^{-1} &  & g_{0}\in G_{0}-(St_{G_{0}}(G_{1})\cap
Z(H_{0})) &  \\
e_{G_{1}} &  & g_{0}\in St_{G_{0}}(G_{1})\cap Z(G_{0}) &
\end{array}%
\right.
\end{equation*}%
Clearly, $\phi \varphi =id_{D_{\mathcal{C}}(H_{0},H_{1})}$ and $\varphi \phi
=id_{D_{\mathcal{C}}(G_{0},G_{1})}.$ Thus the homomorphism $\phi $ is an
isomorphism.
\end{proof}

\begin{proposition}
If $G$ and $H$ are finite crossed modules, then $Aut_{_{\mathcal{C}%
}}(G)\cong Aut_{_{\mathcal{C}}}(H).$
\end{proposition}

\begin{proof}
It can be easily checked by a similar way of Theorem 4.1 in \cite{YA}.
\end{proof}

\begin{corollary}
Let $G:G_{1}\overset{d_{G}}{\longrightarrow }G_{0}$ and $H:H_{1}\overset{%
d_{H}}{\longrightarrow }H_{0}$ be two finite isoclinic crossed modules. Then
$Act_{\mathcal{C}}(G)\cong Act_{\mathcal{C}}(H).$
\end{corollary}

\begin{example}
Let $M$ $=Kl_{4}$ and $N=C_{32}.$ The crossed modules $M\overset{id}{%
\longrightarrow }M$ and $N\overset{id}{\longrightarrow }N$ are isoclinic but
their actors $(Aut(M),Aut(M),\Delta )\ncong (Aut(N),Aut(N),\Delta ).$ On the
other hand, their class preserving actors $(Aut_{\mathcal{C}}(M),Aut_{%
\mathcal{C}}(M),\Delta |)$ and $(Aut_{\mathcal{C}}(N),Aut_{\mathcal{C}%
}(N),\Delta |)$ are isomorphic. ( See Example \ref{enver}, for details. )
\end{example}

In group theory, if $M$ and $N$ are isoclinic groups, then $M$ is nilpotent
(solvable) if and only if $N$ is nilpotent (solvable), and they have the
same nilpotency class (derived length). On the other hand, for crossed
modules, if $G:G_{1}\overset{d}{\longrightarrow }G_{0}$ is nilpotent
(solvable) then all subcrossed modules and all quotient crossed modules of $%
G $ are nilpotent (solvable). Also, if $G/Z(G)$ is nilpotent (solvable),
then $G$ is nilpotent (solvable). So, we have the following result:

\begin{proposition}
Let $G:G_{1}\overset{d_{G}}{\longrightarrow }G_{0}$ and $H:H_{1}\overset{%
d_{H}}{\longrightarrow }H_{0}$ be two isoclinic crossed modules.
\end{proposition}

(i) $G$ is nilpotent (solvable) crossed module if and only if $H$ is
nilpotent (solvable).

(ii) If $G$ and $H$ are nilpotent (solvable) and both nontrivial, then they
have the same nilpotency class (derived length).

\begin{remark}
When we consider the groups as crossed modules, then we recover classical
results for isoclinic groups. In fact, if $M\overset{id}{\longrightarrow }M$
and $N\overset{id}{\longrightarrow }N$ are isoclinic crossed modules then we
find that $M$ and $N$ are isoclinic. On the other hand, let $N$, $M$ are
finitely generated groups and $N^{\prime }\trianglelefteq N,$ $M^{\prime
}\trianglelefteq M.$ Then the isoclinism of inclusion crossed modules {%
\xymatrix {N^{\prime} \ar@{^{(}->}[r]^{inc.} &N}} and {\xymatrix {M^{\prime}
\ar@{^{(}->}[r]^{inc.} &M}} give rise to the isoclinism between the pair
groups $(N^{\prime },N)$ and $(M^{\prime },M).$ Also, the converse is true,
that is, if $(G_{1},G_{0})$ and $(H_{1},H_{0})$ are isoclinic pair groups,
then the resulting inclusion crossed modules {\xymatrix {G_{1}
\ar@{^{(}->}[r]^{inc.} &G_{0}}} and {\xymatrix {H_{1} \ar@{^{(}->}[r]^{inc.}
&H_{0}}} are isoclinic. (See \cite{AFT}, for the definition of isoclinic
pair groups)
\end{remark}

\section{Computer Implementations}

GAP is an open-source system for discrete computational algebra. The system
consists a library of mathematical algorithm implementations, a database
about some algebraic properties of small order groups, vector spaces,
modules, algebras, graphs, codes, designs etc. and some character tables of
these algebraic structures. The system has world wide usage in the area of
education and scientific researches. GAP is free software and user
contributions to the system are supported. These contributions are organized
in a form of GAP packages and are distributed together with the system.
Contributors can submit additional packages for inclusion after a reviewing
process.

Since, no standard GAP function yet exist for checking that two groups $M$
and $N$ are isoclinic or not, first we add the function \texttt{%
IsIsoclinicGroup(M,N)}. In the following GAP session it is seen that the
dihedral group with order 8 and the quaternion group are isoclinic. Notice
that two isoclinic groups may have different orders.

\begin{Verbatim}[frame=single, fontsize=\small, commandchars=\\\{\}]
\textcolor{blue}{gap> Q8 := QuaternionGroup(8);}
<pc group of size 8 with 3 generators>
\textcolor{blue}{gap> D8 := DihedralGroup(8);}
<pc group of size 8 with 3 generators>
\textcolor{blue}{gap> IsIsoclinicGroup(Q8,D8);}
true
\end{Verbatim}

In \cite{RJ}, the 115 isoclinism families induced from 2328 groups with
order 128 and their basic properties were given. We add the function \texttt{%
IsoclinismFamily(M)} to determine the isoclinism classes of the group $M$
with order n. We give a table consisting isoclinism classes of 1543 group
with order 192 in Appendix B, by using this function.

The computer applications of crossed modules were given by Alp and Wensley,
in \cite{AW}, by the shared package XMod. To add a function checks if any
two crossed modules are isoclinic or not, first we need to define the
functions; factor crossed modules, commutator crossed modules which haven't
implemented in the XMod package. Also we redefine a function to find center
of a given crossed module, to make compatible with the other defined
functions.

\subsection{Implementations for centers of crossed modules}

Up to the definition of isoclinic crossed modules, we first need to
construct the center of a given crossed module by using Definition \ref{09}
which is called by \texttt{CenterXMod(XM).} The step-by-step construction of
this function is given as follows;

\textbf{Step 1}: We added the function \texttt{G1G0(XM)}, for computing the
subgroup $G_{1}^{G_{0}},$ the fixed point of $G_{1},$ induced from the
crossed module $XM:G_{1}\overset{d}{\longrightarrow }G_{0}.$
\begin{Verbatim}[frame=single, fontsize=\small, commandchars=\\\{\}]
\textcolor{blue}{gap> G1G0(XM);}
<pc group of size 2 with 1 generators>
\textcolor{blue}{gap> IsSubgroup(Source(XM),last);}
true
\end{Verbatim}

\textbf{Step 2}: We added the function \texttt{StG0G1(XM)}, for computing
the subgroup $St_{G_{0}}G_{1},$ the stabilizer of $G_{1}$ in $G_{0},$
induced from the crossed module $XM:G_{1}\overset{d}{\longrightarrow }G_{0}.$
\begin{Verbatim}[frame=single, fontsize=\small, commandchars=\\\{\}]
\textcolor{blue}{gap> StG0T(XM);}
<pc group of size 1 with 0 generators>
\textcolor{blue}{gap> IsSubgroup(Range(XM),last);}
true
\end{Verbatim}

\textbf{Step 3}: By the definition of center of a crossed module $XM:G_{1}%
\overset{d}{\longrightarrow }G_{0}$ given in Definition \ref{09} $,$ we
added the function \texttt{CenterXMod(XM),} for computing the center.
\begin{Verbatim}[frame=single, fontsize=\small, commandchars=\\\{\}]
\textcolor{blue}{gap> ZXM := CenterXMod(XM);}
[Group( [ f4 ] )->Group( <identity> of ... )]
\textcolor{blue}{gap> IsXMod(ZXM);}
true
\textcolor{blue}{gap> IsNormalXMod(XM,ZXM);}
true
\end{Verbatim}

\subsection{Implementations for factor and commutator subcrossed modules}

Since no standard GAP function yet exist for computing the commutator
subcrossed module of a given crossed module, we added the function \texttt{%
CommutatorSubXMod(XM)}. The step-by-step construction of this function is as
follows;

\textbf{Step 1}: We added the function \texttt{DG0G1(XM), }for computing the
subgroup $D_{G_{0}}(G_{1})$ induced from the crossed module $XM:G_{1}\overset%
{d}{\longrightarrow }G_{0}$ defined in Definition \ref{11}.
\begin{Verbatim}[frame=single, fontsize=\small, commandchars=\\\{\}]
\textcolor{blue}{gap> DerivedSubgroup(Range(XM));}
Group([  ])
\textcolor{blue}{gap> DG0G1(XM);}
<pc group of size 4 with 1 generators>
\textcolor{blue}{gap> IsSubgroup(Source(XM),last);}
true
\end{Verbatim}

\textbf{Step 2}: We added the function \texttt{DerivedSubXMod(XM)} for
computing the subcrossed module $[XM,XM]$ defined in Definition \ref{11}.
\begin{Verbatim}[frame=single, fontsize=\small, commandchars=\\\{\}]
\textcolor{blue}{gap> KM := DerivedSubXMod(XM);}
[Group( [ f3 ] )->Group( <identity> of ... )]
\textcolor{blue}{gap> IsSubXMod(XM,KM);}
true
\textcolor{blue}{gap> IsNormal(XM,KM);}
true
\end{Verbatim}

Additionally, we added the function \texttt{FactorXMod(XM,NM)}, for
computing the quotient crossed modules.
\begin{Verbatim}[frame=single, fontsize=\small, commandchars=\\\{\}]
\textcolor{blue}{gap> FactorXMod(XM,NM);}
[Group( [ f1, <identity> of ..., <identity> of ... ] )->Group( [ f2, f2 ] )]
\textcolor{blue}{gap> IsXMod(last);}
true
\end{Verbatim}

\subsection{Implementations for isoclinic crossed modules}

We have added a function \texttt{IsIsoclinicXMod(XM1,XM2)}, for checking two
crossed modules are isoclinic or not. Step-by-step construction of this
function is as follows;

\textbf{Step 1}: First of all, we needed a function for isomorphism of
crossed modules and so we added the function \texttt{%
IsIsomorphicXMod(XM1,XM2)}. In the following GAP session, it is proved that
the constructed crossed modules XM and XM2 are not isomorphic.
\begin{Verbatim}[frame=single, fontsize=\small, commandchars=\\\{\}]
\textcolor{blue}{gap> C2 := Cat1(32,9,1);}
[(C8 x C2) : C2=>Group( [ f2, f2 ] )]
\textcolor{blue}{gap> XM2 := XMod(C2);}
[Group( [ f1*f2*f3, f3, f4, f5 ] )->Group( [ f2, f2 ] )]
\textcolor{blue}{gap> IsIsomorphicXMod(XM,XM2);}
false
\end{Verbatim}

\textbf{Step 2}: We determined all isomorphisms between the factor crossed
modules $XM1/Z(XM1)$ and $XM2/Z(XM2)$. If there is no such isomorphism, we
arrange the function \newline
\texttt{IsIsoclinicXMod(XM1,XM2)} to make its output \texttt{false}.
\begin{Verbatim}[frame=single, fontsize=\small, commandchars=\\\{\}]
\textcolor{blue}{gap> ZXM2 := CenterXMod(XM2);;}
\textcolor{blue}{gap> IsIsomorphicXMod(FactorXMod(XM,ZXM),FactorXMod(XM2,ZXM2));}
true
\end{Verbatim}

\textbf{Step 3}: We continued the same procedure given in Step 2 for the
commutator subcrossed modules.
\begin{Verbatim}[frame=single, fontsize=\small, commandchars=\\\{\}]
\textcolor{blue}{gap> KM2 := DerivedSubXMod(XM2);;}
\textcolor{blue}{gap> IsIsomorphicXMod(KM,KM2);}
true
\end{Verbatim}

\textbf{Step 4:} Then, after determining the existence of two isomorphism
given in Step 2 and Step 3, we arrange \texttt{IsIsoclinicXMod(XM1,XM2)} to
give out put \texttt{true} if the isomorphisms make the diagrams (1) and (2)
in Definition \ref{04} commutative.
\begin{Verbatim}[frame=single, fontsize=\small, commandchars=\\\{\}]
\textcolor{blue}{gap> IsIsoclinicXMod(XM,XM2);}
true
\textcolor{blue}{gap> Size(XM2);}
[ 16, 2 ]
\end{Verbatim}

\begin{remark}
This GAP session shows that two crossed modules whose orders are different
can be isoclinic as it is the case for groups.
\end{remark}

\section{Character Tables}

For determining isoclinism families of order $[n,m],$ we added a function
\texttt{AllXMods(n,m)} to find all crossed modules of order $[n,m].$
\begin{Verbatim}[frame=single, fontsize=\small, commandchars=\\\{\}]
\textcolor{blue}{gap> list := AllXMods(4,4);;}
\textcolor{blue}{gap> Length(list);}
60
\end{Verbatim}

Then we added the function \texttt{AllXModsByIso(list)} to choose one
representative from all isomorphism families and construct the new list of
crossed modules. Naturally, in the new list, there is no isomorphic crossed
modules.
\begin{Verbatim}[frame=single, fontsize=\small, commandchars=\\\{\}]
\textcolor{blue}{gap> ilist := AllXModsByIso(list);;}
\textcolor{blue}{gap> Length(ilist);}
18
\end{Verbatim}

We added the function \texttt{IsoclinismXModFamily(XM,ilist)}, to get the
isoclinism families for a given order [n,m].

There are two isoclinism families of crossed modules of order [4,4] which is
given in the following GAP session.
\begin{Verbatim}[frame=single, fontsize=\small, commandchars=\\\{\}]
\textcolor{blue}{gap> IsoclinismXModFamily(iso_list[3],ilist);}
[ 1, 3, 4, 6, 8, 10, 12, 14, 16, 18 ]
\textcolor{blue}{gap> IsoclinismXModFamily(iso_list[2],ilist);}
[ 2, 5, 7, 9, 11, 13, 15, 17 ]
\end{Verbatim}

Now, we introduce the notion \textquotedblleft rank\textquotedblright\ and
\textquotedblleft middle length\textquotedblright\ of a crossed module.

\begin{definition}
Let $G:G_{1}\overset{d}{\longrightarrow }G_{0}$ be a finite crossed module.
Then the pair%
\begin{equation*}
\begin{array}{c}
\left( \log _{2}\left\vert G_{1}^{G_{0}}\cap D_{G_{0}}(G_{1})\right\vert ,%
\text{ }\log _{2}\left\vert (St_{G_{0}}(G_{1})\cap Z(G_{0}))\cap \lbrack
G_{0},G_{0}]\right\vert \right) \\
+ \\
\left( \log _{2}\left\vert G_{1}/G_{1}^{G_{0}}\right\vert ,\text{ }\log
_{2}\left\vert G_{0}/(St_{G_{0}}(G_{1})\cap Z(G_{0}))\right\vert \right)%
\end{array}%
\end{equation*}%
is called the \textit{rank} of $G.$ Also the pair%
\begin{equation*}
\left( \log _{2}\left\vert D_{G_{0}}(G_{1})/(G_{1}^{G_{0}}\cap
D_{G_{0}}(G_{1}))\right\vert ,\text{ }\log _{2}\left\vert
[G_{0},G_{0}]/((St_{G_{0}}(G_{1})\cap Z(G_{0}))\cap \lbrack
G_{0},G_{0}])\right\vert \right)
\end{equation*}%
is called the \textit{middle length} of $G.$
\end{definition}

The rank, middle length, nilpotency class and lower central series of the
crossed modules in the same isoclinism family are equal. For computing these
gadgets (these can be thought as the gadgets for correcting our isoclinism
definition) we added the functions \texttt{RankOfXMod(XM)}, \texttt{%
MiddleLengthOfXMod(XM)}, \texttt{LowerCentralSeriesOfXMod(XM)} and \newline
\texttt{NilpotencyClassOfXMod(XM)}.
\begin{Verbatim}[frame=single, fontsize=\small, commandchars=\\\{\}]
\textcolor{blue}{gap> RankOfXMod(XM);}
[ 3, 1 ]
\textcolor{blue}{gap> MiddleLengthOfXMod(XM);}
[ 1, 0 ]
\textcolor{blue}{gap> LowerCentralSeriesOfXMod(XM);}
[ [Group( [ f1*f2*f3, f3, f4 ] )->Group( [ f2, f2 ] )],
  [Group( [ f3 ] )->Group( <identity> of ... )],
  [Group( [ f4 ] )->Group( <identity> of ... )],
  [Group( <identity> of ... )->Group( <identity> of ... )] ]
\textcolor{blue}{gap> Length(last);}
4
\textcolor{blue}{gap> NilpotencyClassOfXMod(XM);}
3
\end{Verbatim}

\begin{example}
The groups of order $8$ has $5$ isomorphism classes and $2$ isoclinism
families which are listed as follows;%

\begin{tabular}{cccccccc}
\multicolumn{8}{c}{Table I} \\
\multicolumn{8}{c}{Number of Groups in Each Isoclinism Family} \\
\multicolumn{8}{c}{and Some Family Invariants} \\ \hline\hline
Family & Numbers & Represent. & Rank & Middle Length & Nilpotency Class & $%
G/Z$ & $\gamma _{2}(G)$ \\ \hline
1 & 3 & [8,1] & 0 & 0 & 1 & [1,1] &  \\
2 & 2 & [8,3] & 3 & 0 & 2 & [4,2] & [2,1] \\ \hline
&  &  &  &  & & &%
\end{tabular}%

By using these 5 isomorphism class, we get 9008 crossed modules with order
[8,8], 294 isomorphism classes and 20 isoclinism families.

The following informations is listed in Table II for each isoclinism family;

\begin{enumerate}
\item the number of crossed modules in the family,

\item the rank of crossed modules in the family,

\item the middle length of crossed modules in the family,

\item the nilpotency class, $c>0,$ of the crossed modules in the family.

\item the size of the central quotient, $XM/Z(XM)$, of a crossed modules $XM$
in the family.

\item the size of the non-trivial or non-repeatedly terms, $\gamma
_{2}(XM),...,\gamma _{c}(XM),$ of the lower central series of $XM;$ here $%
\gamma _{1}(XM)=XM$ and $\gamma _{i+1}(XM)=[\gamma _{i}(XM),XM],$ for $1\leq
i\leq c.$
\end{enumerate}

\begin{tabular}{cccccccc}
\multicolumn{8}{c}{Table II} \\
\multicolumn{8}{c}{Number of Crossed Modules in Each Isoclinism Family} \\
\multicolumn{8}{c}{and Some Family Invariants} \\ \hline\hline
Fam. & Num. & Rank & M. L. & Class & $\left\vert XM/Z(XM)\right\vert $ & $%
\left\vert \gamma _{2}(XM)\right\vert $ & $\left\vert \gamma
_{3}(XM)\right\vert $ \\ \hline
1 & 37 & [0,0] & [0,0] & 1 & [1,1] &  &  \\
2 & 79 & [2,1] & [0,0] & 2 & [2,2] & [2,1] &  \\
3 & 18 & [3,1] & [1,0] & 3 & [4,2] & [4,1] & [2,1] \\
4 & 8 & [3,2] & [1,0] & 3 & [4,4] & [4,1] & [2,1] \\
5 & 14 & [0,3] & [0,0] & 2 & [1,4] & [1,2] &  \\
6 & 42 & [2,3] & [0,0] & 2 & [2,4] & [2,2] &  \\
7 & 12 & [3,3] & [1,0] & 3 & [4,4] & [4,2] & [2,1] \\
8 & 8 & [3,3] & [1,0] & 3 & [4,4] & [4,2] & [2,1] \\
9 & 4 & [3,3] & [1,0] & 3 & [4,4] & [4,2] & [2,1] \\
10 & 4 & [3,2] & [1,0] & 3 & [4,4] & [4,1] & [2,1] \\
11 & 10 & [3,2] & [0,0] & 2 & [2,4] & [4,1] &  \\
12 & 15 & [3,2] & [0,0] & 2 & [4,4] & [2,1] &  \\
13 & 10 & [3,3] & [0,0] & 2 & [2,4] & [4,2] &  \\
14 & 2 & [3,3] & [1,1] & 3 & [4,8] & [4,2] & [2,1] \\
15 & 15 & [3,3] & [0,0] & 2 & [4,4] & [2,2] &  \\
16 & 6 & [2,3] & [0,0] & 2 & [2,4] & [2,2] &  \\
17 & 2 & [3,3] & [1,1] & 3 & [4,8] & [4,2] & [2,1] \\
18 & 2 & [3,3] & [0,0] & 2 & [4,4] & [2,2] &  \\
19 & 2 & [3,2] & [0,0] & 2 & [4,4] & [2,1] &  \\
20 & 4 & [3,2] & [1,0] & 3 & [4,4] & [4,1] & [2,1] \\ \hline
&  &  &  &  & & &%
\end{tabular}%

\end{example}

\newpage
\begin{example}
The group of order 18 has 5 isomorphism classes and 4 isoclinism families.%

\begin{tabular}{cccccccc}
\multicolumn{8}{c}{Table III} \\
\multicolumn{8}{c}{Number of Groups in Each Isoclinism Family} \\
\multicolumn{8}{c}{and Some Family Invariants} \\ \hline\hline
Fam. & Num. & Rep. & Rank & M. L. & Class & $G/Z$ & $\gamma _{2}(G)$ \\
\hline
1 & 1 & [18,1] & 4.17 & 3.17 & 0 & [18,1] & [9,1] \\
2 & 2 & [18,2] & 0.00 & 0.00 & 1 & [1,1] &  \\
3 & 1 & [18,3] & 2.58 & 1.58 & 0 & [6,1] & [3,1] \\
4 & 1 & [18,4] & 4.17 & 3.17 & 0 & [18,4] & [9,2] \\ \hline
&  &  &  &  & & &%
\end{tabular}%

By using these 5 isomorphism classes, we get 2222 crossed modules with order
[18,18], 97 isomorphism classes, 46 isoclinism families.

\begin{tabular}{cccccccc}
\multicolumn{8}{c}{Table IV} \\
\multicolumn{8}{c}{Number of Crossed Modules in Each Isoclinism Family} \\
\multicolumn{8}{c}{and Some Family Invariants} \\ \hline\hline
Fam. & Num. & Rank & M. L. & Class & $\left\vert XM/Z(XM)\right\vert $ & $%
\left\vert \gamma _{2}(XM)\right\vert $ & $\left\vert \gamma
_{3}(XM)\right\vert $ \\ \hline
1 & 1 & [4.17,4.17] & [3.17,3.17] & 0 & [18,18] & [9,9] &  \\
2 & 2 & [0.00,4.17] & [0.00,3.17] & 0 & [1,18] & [1,9] &  \\
3 & 1 & [3.17,4.17] & [3.17,3.17] & 0 & [9,18] & [9,9] &  \\
4 & 1 & [3.17,4.17] & [3.17,3.17] & 0 & [9,18] & [9,9] &  \\
5 & 1 & [3.17,4.17] & [3.17,3.17] & 0 & [9,18] & [9,9] &  \\
6 & 20 & [0.00,0.00] & [0.00,0.00] & 1 & [1,1] &  &  \\
7 & 2 & [3.17,2.58] & [3.17,0.00] & 0 & [9,6] & [9,1] &  \\
8 & 16 & [3.17,1.58] & [0.00,0.00] & 2 & [3,3] & [3,1] &  \\
9 & 2 & [3.17,1.00] & [3.17,0.00] & 0 & [9,2] & [9,1] &  \\
10 & 4 & [0.00,2.58] & [0.00,1.58] & 0 & [1,6] & [1,3] &  \\
11 & 1 & [3.17,4.17] & [3.17,1.58] & 0 & [9,18] & [9,3] &  \\
12 & 2 & [3.17,4.17] & [0.00,1.58] & 0 & [3,18] & [3,3] & [1,3] \\
13 & 1 & [3.17,2.58] & [3.17,1.58] & 0 & [9,6] & [9,3] &  \\
14 & 1 & [3.17,4.17] & [3.17,1.58] & 0 & [9,18] & [9,3] &  \\
15 & 1 & [3.17,2.58] & [3.17,1.58] & 0 & [9,6] & [9,3] &  \\ \hline
&  &  &  &  & \multicolumn{3}{r}{\textit{table continued}}%
\end{tabular}%

\begin{tabular}{cccccccc}
\hline\hline
Fam. & Num. & Rank & M. L. & Class & $\left\vert XM/Z(XM)\right\vert $ & $%
\left\vert \gamma _{2}(XM)\right\vert $ & $\left\vert \gamma
_{3}(XM)\right\vert $ \\
16 & 2 & [0.00,4.17] & [0.00,3.17] & 0 & [1,18] & [1,9] &  \\
17 & 1 & [3.17,4.17] & [3.17,3.17] & 0 & [9,18] & [9,9] &  \\
18 & 1 & [3.17,4.17] & [3.17,3.17] & 0 & [9,18] & [9,9] &  \\
19 & 2 & [2.58,2.58] & [1.58,1.58] & 0 & [6,6] & [3,3] &  \\
20 & 1 & [4.17,4.17] & [3.17,3.17] & 0 & [18,18] & [9,9] &  \\
21 & 1 & [1.58,4.17] & [1.58,3.17] & 0 & [3,18] & [3,9] &  \\
22 & 1 & [3.17,4.17] & [1.58,3.17] & 0 & [9,18] & [3,9] &  \\
23 & 1 & [3.17,4.17] & [1.58,3.17] & 0 & [9,18] & [9,9] &  \\
24 & 1 & [3.17,4.17] & [3.17,3.17] & 0 & [9,18] & [9,9] &  \\
25 & 1 & [1.58,4.17] & [1.58,3.17] & 0 & [3,18] & [3,9] &  \\
26 & 1 & [3.17,4.17] & [1.58,3.17] & 0 & [3,18] & [9,9] &  \\
27 & 1 & [3.17,4.17] & [3.17,3.17] & 0 & [9,18] & [9,9] &  \\
28 & 4 & [1.58,1.00] & [1.58,0.00] & 0 & [3,2] & [3,1] &  \\
29 & 2 & [3.17,2.58] & [3.17,0.00] & 0 & [9,6] & [9,1] &  \\
30 & 2 & [3.17,1.00] & [3.17,0.00] & 0 & [9,2] & [9,1] &  \\
31 & 2 & [1.58,2.58] & [1.58,1.58] & 0 & [3,6] & [3,3] &  \\
32 & 1 & [3.17,4.17] & [3.17,1.58] & 0 & [9,18] & [9,3] &  \\
33 & 1 & [3.17,2.58] & [3.17,1.58] & 0 & [9,6] & [9,3] &  \\
34 & 2 & [3.17,2.58] & [1.58,1.58] & 0 & [9,6] & [3,3] &  \\
35 & 1 & [3.17,2.58] & [1.58,1.58] & 0 & [3,6] & [9,3] &  \\
36 & 2 & [1.58,2.58] & [1.58,1.58] & 0 & [3,6] & [3,3] &  \\
37 & 1 & [3.17,4.17] & [3.17,1.58] & 0 & [9,18] & [9,3] &  \\
38 & 1 & [3.17,2.58] & [3.17,1.58] & 0 & [9,6] & [9,3] &  \\
39 & 1 & [1.58,4.17] & [1.58,3.17] & 0 & [3,18] & [3,9] &  \\
40 & 1 & [3.17,4.17] & [1.58,3.17] & 0 & [9,18] & [3,9] &  \\
41 & 1 & [3.17,4.17] & [1.58,3.17] & 0 & [3,18] & [9,9] &  \\
42 & 1 & [3.17,4.17] & [3.17,3.17] & 0 & [9,18] & [9,9] &  \\
43 & 1 & [1.58,4.17] & [1.58,3.17] & 0 & [3,18] & [3,9] &  \\
44 & 1 & [3.17,4.17] & [1.58,3.17] & 0 & [3,18] & [9,9] &  \\
45 & 1 & [3.17,4.17] & [3.17,3.17] & 0 & [9,18] & [9,9] &  \\
46 & 1 & [3.17,4.17] & [3.17,3.17] & 0 & [9,18] & [9,9] &  \\ \hline
&  &  &  &  & & &%
\end{tabular}%
\end{example}

\section*{Appendix A}

\textit{Well definition of the maps }$c_{1}$\textit{\ and }$c_{0}$\textit{\
used in Definition \ref{04}:}

Let $(g_{1}G_{1}^{G_{0}},g_{0}St_{G_{0}}(G_{1})\cap
Z(G_{0}))=(a_{1}G_{1}^{G_{0}},a_{0}St_{G_{0}}(G_{1})\cap Z(G_{0})),$ then $%
g_{1}^{-1}(a_{1})\in G_{1}^{G_{0}}\subseteq Z(G_{1})$ and $%
g_{0}^{-1}(a_{0})\in St_{G_{0}}(G_{1})\cap Z(G_{0})\subseteq Z(G_{0}).$ Then
we have

\begin{equation*}
\begin{array}{lll}
^{g_{0}}g_{1}g_{1}^{-1} & = & (^{g_{0}}g_{1})g_{1}^{-1}(a_{1})(a_{1})^{-1}
\\
& = & (^{g_{0}}g_{1})^{g_{0}}(g_{1}^{-1}a_{1})(a_{1})^{-1}\text{ }(\because
g_{1}^{-1}a_{1}\in G_{1}^{G_{0}}) \\
& = & ^{g_{0}}(g_{1}g_{1}^{-1}a_{1})(a_{1})^{-1}\text{ }(\because
(^{g_{0}}gg_{1})=(^{g_{0}}g)(^{g_{0}}g_{1})) \\
& = & (^{g_{0}}a_{1})(a_{1})^{-1} \\
& = & ^{(g_{0}^{-1}a_{0})}(^{g_{0}}a_{1})(a_{1})^{-1}\text{ }(\because
g_{0}^{-1}a_{0}\in St_{G_{0}}(G_{1})) \\
& = & (^{(g_{0}^{-1}a_{0})g_{0}}a_{1})(a_{1})^{-1}\text{ }(\because
^{(gg^{\prime })}g_{1}=\text{ }^{^{g}g^{\prime }}g_{1}) \\
& = & (^{g_{0}(g_{0}^{-1}a_{0})}a_{1})(a_{1})^{-1}\text{ }(\because
g_{0}^{-1}a_{0}\in Z(G_{0})) \\
& = & (^{a_{0}}a_{1})(a_{1})^{-1}%
\end{array}%
\end{equation*}%
which gives the well-definition of $c_{1}.$

Similarly, let $\overline{g_{0}},\overline{a_{0}},\overline{g_{0}^{\prime }},%
\overline{a_{0}^{\prime }}\in \overline{G_{0}}$. If $(\overline{g_{0}},%
\overline{g_{0}^{\prime }})=(\overline{a_{0}},\overline{a_{0}^{\prime }}),$
then $g_{0}^{-1}a_{0}\in St_{G_{0}}(G_{1})\cap Z(G_{0})\subseteq Z(G_{0})$
and $(g_{0}^{\prime })^{-1}a_{0}^{\prime }\in St_{G_{0}}(G_{1})\cap
Z(G_{0})\subseteq Z(G_{0}).$ Then, we have%
\begin{eqnarray*}
\lbrack g_{0},g_{0}^{\prime }] &=&g_{0}g_{0}^{\prime
}g_{0}^{-1}(g_{0}^{\prime })^{-1} \\
&=&g_{0}g_{0}^{\prime }g_{0}^{-1}(a_{0}a_{0}^{-1})(g_{0}^{\prime })^{-1} \\
&=&g_{0}g_{0}^{\prime }(g_{0}^{-1}a_{0})a_{0}^{-1}(g_{0}^{\prime })^{-1} \\
&=&g_{0}(g_{0}^{-1}a_{0})g_{0}^{\prime }a_{0}^{-1}(g_{0}^{\prime })^{-1}%
\text{ }(\because g_{0}^{-1}a_{0}\in Z(G_{0})) \\
&=&a_{0}g_{0}^{\prime }a_{0}^{-1}(g_{0}^{\prime })^{-1} \\
&=&a_{0}g_{0}^{\prime }a_{0}^{-1}(g_{0}^{\prime })^{-1}(a_{0}^{\prime
}(a_{0}^{\prime })^{-1}) \\
&=&a_{0}g_{0}^{\prime }a_{0}^{-1}((g_{0}^{^{\prime }})^{-1}a_{0}^{\prime
})(a_{0}^{\prime })^{-1} \\
&=&a_{0}g_{0}^{\prime }((g_{0}^{\prime })^{-1}a_{0}^{\prime
})a_{0}^{-1}(a_{0}^{\prime })^{-1}\text{ }(\because (g_{0}^{\prime
})^{-1}a_{0}^{\prime }\in Z(G_{0})) \\
&=&a_{0}a_{0}^{\prime }a_{0}^{-1}(a_{0}^{\prime })^{-1} \\
&=&[a_{0},a_{0}^{\prime }]
\end{eqnarray*}%
which gives the well-definition of $c_{0.}$

Well definition of the maps $c_{1}^{\prime }$ and $c_{0}^{\prime }$ can be
shown by a similar way.

\textit{Proof of Proposition \ref{06}:}

(a) i) Let $\delta ,\delta ^{\prime }\in D_{\mathcal{C}}(G_{0},G_{1}).$ We
first show that $\delta \delta ^{\prime }\in D_{\mathcal{C}}(G_{0},G_{1}).$
Since $\delta ,\delta ^{\prime }\in D_{\mathcal{C}}(G_{0},G_{1}),$ there
exist $g_{1},g_{1}^{\prime }\in G_{1}$ such that $\delta (g_{0})=g_{1}$ $%
^{g_{0}}g_{1}^{-1}$and $\delta ^{\prime }(g_{0})=(g_{1}^{\prime
})^{g_{0}}g_{1}^{\prime },$ for all $g_{0}\in G_{0},$ $g_{1},g_{1}^{\prime
}\in G_{1}.$ Then,%
\begin{equation*}
\begin{array}{lll}
\delta \delta ^{\prime }(g_{0}) & = & \delta (d\delta ^{\prime
}(g_{0})g_{0})\delta ^{\prime }(g_{0}) \\
& = & \delta (d(\delta ^{\prime }(g_{0}))g_{0})\delta ^{\prime }(g_{0}) \\
& = & \delta (d(\delta ^{\prime }(g_{0}))^{d(\delta ^{\prime
}(g_{0}))}\delta (g_{0})\delta ^{\prime }(g_{0})\text{ }(\because \delta \in
D(G_{0},G_{1})) \\
& = & \delta (d(\delta ^{\prime }(g_{0}))\delta ^{\prime }(g_{0})\delta
(g_{0})(\delta ^{\prime }(g_{0}))^{-1}\delta ^{\prime }(g_{0})\text{ }%
(\because d\sim \text{crossed module}) \\
& = & g_{1}\text{ }^{d(\delta ^{\prime }(g_{0}))}g_{1}^{-1}(g_{1}^{\prime
})^{g_{0}}(g_{1}^{\prime })^{-1}g_{1}\text{ }^{g_{0}}g_{1}^{-1} \\
& = & g_{1}\delta ^{\prime }(g_{0})g_{1}^{-1}(\delta ^{\prime
}(g_{0}))^{-1}(g_{1}^{\prime })^{g_{0}}(g_{1}^{\prime })^{-1}g_{1}\text{ }%
^{g_{0}}g_{1}^{-1} \\
& = & g_{1}((g_{1}^{\prime })^{g_{0}}(g_{1}^{\prime
})^{-1})g_{1}^{-1}((g_{1}^{\prime })^{g_{0}}(g_{1}^{\prime
})^{-1})^{-1}(g_{1}^{\prime })^{g_{0}}(g_{1}^{\prime })^{-1}g_{1}\text{ }%
^{g_{0}}g_{1}^{-1} \\
& = & g_{1}g_{1}^{\prime }\text{ }^{g_{0}}(g_{1}^{\prime
})^{-1}g_{1}^{-1}(^{g_{0}}(g_{1}^{\prime })^{-1})^{-1}(g_{1}^{\prime
})^{-1}(g_{1}^{\prime })(^{g_{0}}(g_{1}^{\prime })^{-1})g_{1}\text{ }%
^{g_{0}}g_{1}^{-1} \\
& = & g_{1}g_{1}^{\prime }\text{ }^{g_{0}}(g_{1}^{\prime
})^{-1}g_{1}^{-1}(^{g_{0}}(g_{1}^{\prime
})^{-1})^{-1}(^{g_{0}}(g_{1}^{\prime })^{-1})g_{1}\text{ }^{g_{0}}g_{1}^{-1}
\\
& = & g_{1}g_{1}^{\prime }\text{ }^{g_{0}}(g_{1}^{\prime
})^{-1}g_{1}^{-1}g_{1}\text{ }^{g_{0}}g_{1}^{-1} \\
& = & g_{1}g_{1}^{\prime }\text{ }^{g_{0}}((g_{1}^{\prime })^{-1}g_{1}^{-1})
\\
& = & (g_{1}g_{1}^{\prime })^{g_{0}}(g_{1}g_{1}^{\prime })^{-1}%
\end{array}%
\end{equation*}%
i.e $\delta \delta ^{\prime }\in D_{\mathcal{C}}(G_{0},G_{1}).$\newline
ii) Let $\delta \in D_{\mathcal{C}}(G_{0},G_{1}).$ Since $\delta \in D_{%
\mathcal{C}}(G_{0},G_{1}),$ there exists $g_{1}\in G_{1}$ such that $\delta
(g_{0})=g_{1}$ $^{g_{0}}g_{1}^{-1},$ for all $g_{0}\in G_{0}$. Define $%
\delta ^{-1}(g_{0})=(g_{1}^{-1})^{g_{0}}g_{1}$, then we have%
\begin{eqnarray*}
\delta \delta ^{-1}(g_{0}) &=&g_{1}g_{1}^{-1}\text{ }%
^{g_{0}}(g_{1}^{-1}g_{1})^{-1} \\
&=&g_{1}g_{1}^{-1}{}^{g_{0}}e_{G_{1}} \\
&=&e_{G_{1}}e_{G_{1}} \\
&=&e_{G_{1}} \\
&=&id_{D_{\mathcal{C}}(G_{0},G_{1})}(g_{0}).
\end{eqnarray*}%
So $D_{\mathcal{C}}(G_{0},G_{1})\leq D(G_{0},G_{1}).$

(b) Let $(\alpha ,\beta ),(\alpha ^{\prime },\beta ^{\prime })\in Aut_{_{%
\mathcal{C}}}(G).$ Since $(\alpha ,\beta ),(\alpha ^{\prime },\beta ^{\prime
})\in Aut_{_{\mathcal{C}}}(G),$ there exist $g_{0},g_{0}^{\prime }\in G_{0}$
such that $\alpha (g_{1})=$ $^{g_{0}}g_{1};$\newline
$\alpha ^{\prime }(g_{1})=$ $^{g_{0}^{\prime }}g_{1};$ $\beta
(h_{0})=g_{0}h_{0}g_{0}^{-1};$ $\beta ^{\prime }(h_{0})=g_{0}^{\prime
}h_{0}(g_{0}^{\prime })^{-1},$ for all $g_{0},h_{0}\in G_{0}$, $g_{1}\in
G_{1}$. Then%
\begin{eqnarray*}
(\alpha \circ \alpha ^{\prime })(g_{1}) &=&\alpha (\alpha ^{\prime }(g_{1}))
\\
&=&\alpha (^{g_{0}}g_{1}) \\
&=&^{^{g_{0}^{\prime }}g_{0}}g_{1} \\
&=&^{(g_{0}^{\prime }g_{0})}g_{1}
\end{eqnarray*}
and%
\begin{eqnarray*}
(\beta \circ \beta ^{\prime })(h_{0}) &=&\beta (\beta ^{\prime }(h_{0})) \\
&=&\beta (g_{0}^{\prime }h_{0}(g_{0}^{\prime })^{-1}) \\
&=&g_{0}(g_{0}^{\prime }h_{0}(g_{0}^{\prime })^{-1})g_{0}^{-1} \\
&=&g_{0}g_{0}^{\prime }h_{0}(g_{0}g_{0}^{\prime })^{-1}.
\end{eqnarray*}
So $(\alpha ,\beta )\circ (\alpha ^{\prime },\beta ^{\prime })\in Aut_{_{%
\mathcal{C}}}(G),$ for all $(\alpha ,\beta ),(\alpha ^{\prime },\beta
^{\prime })\in Aut_{_{\mathcal{C}}}(G).$

Let $(\alpha ,\beta )\in Aut_{_{\mathcal{C}}}(G).$ If we define $(\alpha
,\beta )^{-1}=(\alpha ^{-1},\beta ^{-1})$ by $\alpha ^{-1}(g_{1})=$ $%
^{g_{0}^{-1}}g_{1},$ $\beta ^{-1}(g_{0}^{\prime })=g_{0}^{-1}g_{0}^{\prime
}g_{0},$ for all $g_{0},g_{0}^{\prime }\in G_{0},$ $g_{1}\in G_{1},$ then we
have%
\begin{eqnarray*}
(\alpha \circ \alpha ^{-1})(g_{1}) &=&\alpha (\alpha ^{-1}(g_{1})) \\
&=&\alpha (^{g_{0}^{-1}}g_{1}) \\
&=&^{^{g_{0}}g_{0}^{-1}}g_{1} \\
&=&^{(g_{0}g_{0}^{-1})}g_{1} \\
&=&^{e_{G_{0}}}g_{1} \\
&=&g_{1} \\
&=&id_{G_{1}}(g_{1})
\end{eqnarray*}%
and similarly $\beta \circ \beta ^{-1}=id_{G_{0}}.$ So $Aut_{_{\mathcal{C}%
}}(G)\leq Aut(G).$

\section*{Appendix B}

\begin{tabular}{cccccccccccc}
\multicolumn{12}{c}{Table V} \\
\multicolumn{12}{c}{Number of Groups in Each Isoclinism Family} \\
\multicolumn{12}{c}{and Some Family Invariants} \\ \hline\hline
Fam. & Num. & Rep. & Rank & M. L. & Class & $G/Z$ & $\gamma _{2}(G)$ & $%
\gamma _{3}(G)$ & $\gamma _{4}(G)$ & $\gamma _{5}(G)$ & $\gamma _{6}(G)$ \\
\hline
1 & 19 & [192,1] & 2.58 & 1.58 & 0 & [6,1] & [3,1] &  &  &  &  \\
2 & 11 & [192,2] & 0.00 & 0.00 & 1 & [1,1] &  &  &  &  &  \\
3 & 1 & [192,3] & 7.58 & 6.00 & 0 & [192,3] & [64,2] &  &  &  &  \\
4 & 1 & [192,4] & 7.58 & 4.00 & 0 & [48,3] & [64,19] &  &  &  &  \\
5 & 53 & [192,6] & 4.58 & 1.58 & 0 & [12,4] & [6,2] & [3,1] &  &  &  \\
6 & 3 & [192,7] & 7.58 & 4.58 & 0 & [96,6] & [48,2] & [24,2] & [12,2] & [6,2]
& [3,1] \\
7 & 10 & [192,10] & 7.58 & 2.58 & 0 & [48,14] & [24,9] & [12,5] & [3,1] &  &
\\
8 & 25 & [192,15] & 5.58 & 2.58 & 0 & [24,6] & [12,2] & [6,2] & [3,1] &  &
\\
9 & 5 & [192,25] & 7.58 & 1.58 & 0 & [48,11] & [12,2] & [3,1] &  &  &  \\
10 & 12 & [192,27] & 6.58 & 2.58 & 0 & [48,14] & [12,5] & [6,2] & [3,1] &  &
\\ \hline
&  &  &  &  &  &  &  &  & \multicolumn{3}{r}{\textit{table continued}}%
\end{tabular}%

\begin{tabular}{cccccccccccc}
\hline\hline
Fam. & Num. & Rep. & Rank & M. L. & Class & $G/Z$ & $\gamma _{2}(G)$ & $%
\gamma _{3}(G)$ & $\gamma _{4}(G)$ & $\gamma _{5}(G)$ & $\gamma _{6}(G)$ \\
\hline
11 & 4 & [192,30] & 7.58 & 3.58 & 0 & [96,13] & [24,15] & [12,5] & [6,2] &
[3,1] &  \\
12 & 4 & [192,32] & 7.58 & 3.58 & 0 & [96,13] & [24,9] & [12,5] & [6,2] &
[3,1] &  \\
13 & 35 & [192,38] & 5.58 & 2.58 & 0 & [24,8] & [12,2] & [6,2] & [3,1] &  &
\\
14 & 2 & [192,46] & 7.58 & 3.58 & 0 & [96,15] & [24,2] & [12,2] & [6,2] &
[3,1] &  \\
15 & 3 & [192,47] & 7.58 & 3.58 & 0 & [96,16] & [24,2] & [12,2] & [6,2] &
[3,1] &  \\
16 & 14 & [192,48] & 6.58 & 3.58 & 0 & [48,15] & [24,2] & [12,2] & [6,2] &
[3,1] &  \\
17 & 10 & [192,62] & 6.58 & 3.58 & 0 & [48,7] & [24,2] & [12,2] & [6,2] &
[3,1] &  \\
18 & 1 & [192,72] & 7.58 & 3.58 & 0 & [96,24] & [24,2] & [12,2] & [6,2] &
[3,1] &  \\
19 & 3 & [192,75] & 7.58 & 3.58 & 0 & [96,28] & [24,2] & [12,2] & [6,2] &
[3,1] &  \\
20 & 4 & [192,78] & 7.58 & 4.58 & 0 & [96,33] & [48,2] & [24,2] & [12,2] &
[6,2] & [3,1] \\
21 & 11 & [192,84] & 6.58 & 2.58 & 0 & [48,19] & [12,5] & [6,2] & [3,1] &  &
\\
22 & 2 & [192,95] & 7.58 & 3.58 & 0 & [96,41] & [24,15] & [12,5] & [6,2] &
[3,1] &  \\
23 & 7 & [192,96] & 7.58 & 2.58 & 0 & [48,19] & [24,9] & [12,5] & [3,1] &  &
\\
24 & 4 & [192,100] & 7.58 & 3.58 & 0 & [96,41] & [24,9] & [12,5] & [6,2] &
[3,1] &  \\
25 & 3 & [192,122] & 7.58 & 3.58 & 0 & [96,39] & [24,2] & [12,2] & [6,2] &
[3,1] &  \\
26 & 31 & [192,128] & 3.00 & 0.00 & 2 & [4,2] & [2,1] &  &  &  &  \\
27 & 11 & [192,129] & 5.00 & 1.00 & 3 & [16,3] & [4,2] & [2,1] &  &  &  \\
28 & 25 & [192,131] & 4.00 & 1.00 & 3 & [8,3] & [4,1] & [2,1] &  &  &  \\
29 & 7 & [192,133] & 6.00 & 1.00 & 3 & [16,3] & [8,2] & [4,2] &  &  &  \\
30 & 3 & [192,143] & 6.00 & 0.00 & 2 & [16,2] & [4,1] &  &  &  &  \\
31 & 2 & [192,157] & 6.00 & 2.00 & 4 & [32,6] & [8,5] & [4,2] & [2,1] &  &
\\
32 & 4 & [192,159] & 6.00 & 2.00 & 4 & [32,6] & [8,2] & [4,2] & [2,1] &  &
\\
33 & 10 & [192,163] & 5.00 & 2.00 & 4 & [16,7] & [8,1] & [4,1] & [2,1] &  &
\\
34 & 3 & [192,166] & 6.00 & 2.00 & 4 & [32,9] & [8,1] & [4,1] & [2,1] &  &
\\
35 & 1 & [192,171] & 6.00 & 2.00 & 4 & [32,13] & [8,1] & [4,1] & [2,1] &  &
\\
36 & 3 & [192,177] & 6.00 & 3.00 & 5 & [32,18] & [16,1] & [8,1] & [4,1] &
[2,1] &  \\
37 & 2 & [192,180] & 7.58 & 5.58 & 0 & [96,64] & [96,3] &  &  &  &  \\
38 & 2 & [192,182] & 6.58 & 5.58 & 0 & [96,64] & [48,3] &  &  &  &  \\
39 & 11 & [192,183] & 5.58 & 3.58 & 0 & [24,12] & [24,3] &  &  &  &  \\
40 & 1 & [192,184] & 7.58 & 5.58 & 0 & [192,184] & [48,50] &  &  &  &  \\ \hline
&  &  &  &  &  &  &  &  & \multicolumn{3}{r}{\textit{table continued}}%
\end{tabular}%

\begin{tabular}{cccccccccccc}
\hline\hline
Fam. & Num. & Rep. & Rank & M. L. & Class & $G/Z$ & $\gamma _{2}(G)$ & $%
\gamma _{3}(G)$ & $\gamma _{4}(G)$ & $\gamma _{5}(G)$ & $\gamma _{6}(G)$ \\
\hline
41 & 1 & [192,185] & 7.58 & 5.58 & 0 & [192,185] & [48,3] &  &  &  &  \\
42 & 7 & [192,186] & 4.58 & 3.58 & 0 & [24,12] & [12,3] &  &  &  &  \\
43 & 2 & [192,188] & 5.58 & 4.00 & 0 & [48,3] & [16,2] &  &  &  &  \\
44 & 2 & [192,189] & 6.58 & 4.00 & 0 & [48,3] & [32,2] &  &  &  &  \\
45 & 2 & [192,191] & 6.58 & 4.00 & 0 & [96,70] & [16,14] &  &  &  &  \\
46 & 2 & [192,192] & 6.58 & 4.00 & 0 & [96,71] & [16,2] &  &  &  &  \\
47 & 2 & [192,193] & 6.58 & 4.00 & 0 & [96,72] & [16,2] &  &  &  &  \\
48 & 2 & [192,194] & 7.58 & 4.00 & 0 & [96,72] & [32,2] &  &  &  &  \\
49 & 2 & [192,195] & 7.58 & 4.00 & 0 & [96,70] & [32,47] &  &  &  &  \\
50 & 2 & [192,197] & 7.58 & 4.00 & 0 & [96,71] & [32,2] &  &  &  &  \\
51 & 7 & [192,200] & 4.58 & 2.00 & 0 & [12,3] & [8,4] &  &  &  &  \\
52 & 2 & [192,201] & 7.58 & 4.00 & 0 & [96,70] & [32,49] &  &  &  &  \\
53 & 5 & [192,203] & 3.58 & 2.00 & 0 & [12,3] & [4,2] &  &  &  &  \\
54 & 106 & [192,205] & 6.58 & 1.58 & 0 & [24,14] & [12,5] & [3,1] &  &  &
\\
55 & 49 & [192,207] & 5.58 & 1.58 & 0 & [24,14] & [6,2] & [3,1] &  &  &  \\
56 & 36 & [192,215] & 7.58 & 1.58 & 0 & [24,14] & [24,15] & [3,1] &  &  &
\\
57 & 35 & [192,238] & 6.58 & 1.58 & 0 & [24,14] & [12,5] & [3,1] &  &  &  \\
58 & 13 & [192,239] & 7.58 & 2.58 & 0 & [48,36] & [24,9] & [6,2] & [3,1] &
&  \\
59 & 9 & [192,261] & 7.58 & 2.58 & 0 & [48,36] & [24,9] & [6,2] & [3,1] &  &
\\
60 & 13 & [192,269] & 6.58 & 2.58 & 0 & [48,36] & [12,1] & [6,2] & [3,1] &
&  \\
61 & 24 & [192,280] & 7.58 & 2.58 & 0 & [48,36] & [24,9] & [6,2] & [3,1] &
&  \\
62 & 6 & [192,299] & 7.58 & 2.58 & 0 & [96,87] & [12,5] & [6,2] & [3,1] &  &
\\
63 & 2 & [192,300] & 7.58 & 3.58 & 0 & [96,89] & [24,15] & [6,2] & [3,1] &
&  \\
64 & 4 & [192,305] & 7.58 & 3.58 & 0 & [96,89] & [24,9] & [6,2] & [3,1] &  &
\\
65 & 4 & [192,307] & 7.58 & 3.58 & 0 & [96,91] & [24,9] & [6,2] & [3,1] &  &
\\
66 & 30 & [192,315] & 6.58 & 2.58 & 0 & [48,38] & [12,2] & [6,2] & [3,1] &
&  \\
67 & 26 & [192,316] & 6.58 & 2.58 & 0 & [48,38] & [12,2] & [6,2] & [3,1] &
&  \\
68 & 64 & [192,318] & 7.58 & 2.58 & 0 & [48,38] & [24,9] & [6,2] & [3,1] &
&  \\
69 & 24 & [192,319] & 7.58 & 2.58 & 0 & [48,38] & [24,9] & [6,2] & [3,1] &
&  \\
70 & 20 & [192,323] & 7.58 & 2.58 & 0 & [48,38] & [24,9] & [6,2] & [3,1] &
&  \\
\hline
&  &  &  &  &  &  &  &  & \multicolumn{3}{r}{\textit{table continued}}%
\end{tabular}%

\begin{tabular}{cccccccccccc}
\hline\hline
Fam. & Num. & Rep. & Rank & M. L. & Class & $G/Z$ & $\gamma _{2}(G)$ & $%
\gamma _{3}(G)$ & $\gamma _{4}(G)$ & $\gamma _{5}(G)$ & $\gamma _{6}(G)$ \\
\hline
71 & 4 & [192,381] & 7.58 & 3.58 & 0 & [96,89] & [24,9] & [6,2] & [3,1] &  &
\\
72 & 3 & [192,455] & 7.58 & 3.58 & 0 & [96,102] & [24,9] & [6,2] & [3,1] &
&  \\
73 & 2 & [192,467] & 7.58 & 3.58 & 0 & [96,110] & [24,2] & [12,2] & [6,2] &
[3,1] &  \\
74 & 6 & [192,469] & 7.58 & 3.58 & 0 & [96,117] & [24,2] & [12,2] & [6,2] &
[3,1] &  \\
75 & 4 & [192,470] & 7.58 & 3.58 & 0 & [96,117] & [24,2] & [12,2] & [6,2] &
[3,1] &  \\
76 & 21 & [192,523] & 6.58 & 2.58 & 0 & [48,43] & [12,2] & [6,2] & [3,1] &
&  \\
77 & 24 & [192,526] & 7.58 & 2.58 & 0 & [48,43] & [24,9] & [6,2] & [3,1] &
&  \\
78 & 2 & [192,591] & 7.58 & 3.58 & 0 & [96,144] & [24,15] & [6,2] & [3,1] &
&  \\
79 & 24 & [192,592] & 7.58 & 2.58 & 0 & [48,43] & [24,9] & [6,2] & [3,1] &
&  \\
80 & 18 & [192,597] & 7.58 & 2.58 & 0 & [48,43] & [24,9] & [6,2] & [3,1] &
&  \\
81 & 16 & [192,598] & 7.58 & 2.58 & 0 & [48,43] & [24,9] & [6,2] & [3,1] &
&  \\
82 & 4 & [192,620] & 7.58 & 3.58 & 0 & [96,144] & [24,9] & [6,2] & [3,1] &
&  \\
83 & 3 & [192,700] & 7.58 & 3.58 & 0 & [96,137] & [24,9] & [6,2] & [3,1] &
&  \\
84 & 4 & [192,706] & 7.58 & 3.58 & 0 & [96,138] & [24,2] & [12,2] & [6,2] &
[3,1] &  \\
85 & 3 & [192,719] & 7.58 & 3.58 & 0 & [96,145] & [24,9] & [6,2] & [3,1] &
&  \\
86 & 4 & [192,757] & 7.58 & 3.58 & 0 & [96,144] & [24,9] & [6,2] & [3,1] &
&  \\
87 & 4 & [192,758] & 7.58 & 3.58 & 0 & [96,146] & [24,9] & [6,2] & [3,1] &
&  \\
88 & 4 & [192,800] & 7.58 & 3.58 & 0 & [96,160] & [24,9] & [6,2] & [3,1] &
&  \\
89 & 2 & [192,802] & 7.58 & 3.58 & 0 & [96,160] & [24,15] & [6,2] & [3,1] &
&  \\
90 & 35 & [192,812] & 5.00 & 0.00 & 2 & [8,5] & [4,2] &  &  &  &  \\
91 & 10 & [192,825] & 6.00 & 0.00 & 2 & [8,5] & [8,5] &  &  &  &  \\
92 & 13 & [192,850] & 5.00 & 1.00 & 3 & [16,11] & [4,1] & [2,1] &  &  &  \\
93 & 24 & [192,880] & 6.00 & 1.00 & 3 & [16,11] & [8,2] & [2,1] &  &  &  \\
94 & 4 & [192,886] & 6.00 & 2.00 & 3 & [32,27] & [8,2] & [2,1] &  &  &  \\
95 & 2 & [192,890] & 6.00 & 2.00 & 3 & [32,27] & [8,5] & [2,1] &  &  &  \\
96 & 13 & [192,898] & 6.00 & 1.00 & 3 & [16,11] & [8,2] & [2,1] &  &  &  \\
97 & 9 & [192,901] & 6.00 & 1.00 & 3 & [16,11] & [8,2] & [2,1] &  &  &  \\
98 & 3 & [192,904] & 6.00 & 2.00 & 3 & [32,28] & [8,2] & [2,1] &  &  &  \\
99 & 2 & [192,942] & 6.00 & 2.00 & 4 & [32,39] & [8,1] & [4,1] & [2,1] &  &
\\
100 & 10 & [192,945] & 7.58 & 3.58 & 0 & [48,48] & [48,32] & [24,3] &  &  &
\\ \hline
&  &  &  &  &  &  &  &  & \multicolumn{3}{r}{\textit{table continued}}%
\end{tabular}%

\begin{tabular}{cccccccccccc}
\hline\hline
Fam. & Num. & Rep. & Rank & M. L. & Class & $G/Z$ & $\gamma _{2}(G)$ & $%
\gamma _{3}(G)$ & $\gamma _{4}(G)$ & $\gamma _{5}(G)$ & $\gamma _{6}(G)$ \\
\hline
101 & 10 & [192,947] & 6.58 & 3.58 & 0 & [48,48] & [24,3] &  &  &  &  \\
102 & 1 & [192,955] & 7.58 & 5.58 & 0 & [192,955] & [48,50] &  &  &  &  \\
103 & 1 & [192,956] & 7.58 & 5.58 & 0 & [192,956] & [48,3] &  &  &  &  \\
104 & 3 & [192,957] & 7.58 & 4.58 & 0 & [96,187] & [48,31] & [24,13] & [12,3]
&  &  \\
105 & 10 & [192,959] & 6.58 & 3.58 & 0 & [48,48] & [24,13] & [12,3] &  &  &
\\
106 & 3 & [192,962] & 7.58 & 4.58 & 0 & [96,187] & [48,33] & [24,3] &  &  &
\\
107 & 4 & [192,973] & 7.58 & 4.58 & 0 & [96,195] & [48,31] & [24,13] & [12,3]
&  &  \\
108 & 4 & [192,987] & 7.58 & 4.58 & 0 & [96,195] & [48,33] & [24,3] &  &  &
\\
109 & 6 & [192,994] & 6.58 & 2.00 & 0 & [48,49] & [8,5] & [4,2] &  &  &  \\
110 & 6 & [192,998] & 6.58 & 2.00 & 0 & [48,49] & [8,4] &  &  &  &  \\
111 & 5 & [192,1003] & 7.58 & 2.00 & 0 & [48,49] & [16,12] & [8,4] &  &  &
\\
112 & 1 & [192,1008] & 7.58 & 4.00 & 0 & [192,1008] & [16,2] &  &  &  &  \\
113 & 1 & [192,1009] & 7.58 & 4.00 & 0 & [192,1009] & [16,14] &  &  &  &  \\
114 & 3 & [192,1014] & 7.58 & 3.00 & 0 & [96,197] & [16,10] & [8,5] & [4,2]
&  &  \\
115 & 3 & [192,1017] & 7.58 & 3.00 & 0 & [96,197] & [16,13] & [8,4] &  &  &
\\
116 & 1 & [192,1020] & 7.58 & 6.00 & 0 & [192,1020] & [64,192] &  &  &  &
\\
117 & 1 & [192,1021] & 7.58 & 4.00 & 0 & [48,50] & [64,224] &  &  &  &  \\
118 & 1 & [192,1022] & 7.58 & 4.00 & 0 & [48,50] & [64,239] &  &  &  &  \\
119 & 1 & [192,1023] & 7.58 & 6.00 & 0 & [1092,1023] & [64,242] &  &  &  &
\\
120 & 1 & [192,1024] & 7.58 & 4.00 & 0 & [48,50] & [64,242] &  &  &  &  \\
121 & 1 & [192,1025] & 7.58 & 6.00 & 0 & [192,1025] & [64,245] &  &  &  &
\\
122 & 24 & [192,1042] & 7.58 & 1.58 & 0 & [48,51] & [12,5] & [3,1] &  &  &
\\
123 & 20 & [192,1045] & 6.58 & 1.58 & 0 & [48,51] & [6,2] & [3,1] &  &  &
\\
124 & 50 & [192,1049] & 7.58 & 1.58 & 0 & [48,51] & [12,5] & [3,1] &  &  &
\\
125 & 26 & [192,1145] & 7.58 & 1.58 & 0 & [48,51] & [12,5] & [3,1] &  &  &
\\
126 & 55 & [192,1146] & 7.58 & 1.58 & 0 & [48,51] & [12,5] & [3,1] &  &  &
\\
127 & 50 & [192,1148] & 7.58 & 1.58 & 0 & [48,51] & [12,5] & [3,1] &  &  &
\\
128 & 19 & [192,1153] & 7.58 & 1.58 & 0 & [48,51] & [12,5] & [3,1] &  &  &
\\
129 & 3 & [192,1310] & 7.58 & 2.58 & 0 & [96,207] & [12,2] & [6,2] & [3,1] &
&  \\
130 & 6 & [192,1316] & 7.58 & 2.58 & 0 & [96,209] & [12,2] & [6,2] & [3,1] &
&  \\ \hline
&  &  &  &  &  &  &  &  & \multicolumn{3}{r}{\textit{table continued}}%
\end{tabular}%

\begin{tabular}{cccccccccccc}
\hline\hline
Fam. & Num. & Rep. & Rank & M. L. & Class & $G/Z$ & $\gamma _{2}(G)$ & $%
\gamma _{3}(G)$ & $\gamma _{4}(G)$ & $\gamma _{5}(G)$ & $\gamma _{6}(G)$ \\
\hline
131 & 4 & [192,1331] & 7.58 & 2.58 & 0 & [96,209] & [12,2] & [6,2] & [3,1] &
&  \\
132 & 4 & [192,1333] & 7.58 & 2.58 & 0 & [96,209] & [12,2] & [6,2] & [3,1] &
&  \\
133 & 4 & [192,1394] & 7.58 & 2.58 & 0 & [96,219] & [12,2] & [6,2] & [3,1] &
&  \\
134 & 7 & [192,1407] & 5.00 & 0.00 & 2 & [16,14] & [2,1] &  &  &  &  \\
135 & 11 & [192,1423] & 6.00 & 0.00 & 2 & [16,14] & [4,2] &  &  &  &  \\
136 & 15 & [192,1434] & 6.00 & 0.00 & 2 & [16,14] & [4,2] &  &  &  &  \\
137 & 5 & [192,1449] & 6.00 & 0.00 & 2 & [16,14] & [4,2] &  &  &  &  \\
138 & 3 & [192,1465] & 6.00 & 1.00 & 3 & [32,46] & [4,1] & [2,1] &  &  &  \\
139 & 4 & [192,1472] & 7.58 & 3.58 & 0 & [96,226] & [24,13] & [12,3] &  &  &
\\
140 & 4 & [192,1483] & 7.58 & 3.58 & 0 & [96,226] & [24,3] &  &  &  &  \\
141 & 2 & [192,1489] & 7.58 & 5.58 & 0 & [96,227] & [96,203] &  &  &  &  \\
142 & 2 & [192,1491] & 7.58 & 5.58 & 0 & [96,227] & [96,204] &  &  &  &  \\
143 & 2 & [192,1492] & 7.58 & 5.58 & 0 & [96,227] & [96,204] &  &  &  &  \\
144 & 2 & [192,1495] & 6.58 & 5.58 & 0 & [96,227] & [48,50] &  &  &  &  \\
145 & 2 & [192,1505] & 5.58 & 4.00 & 0 & [48,50] & [16,14] &  &  &  &  \\
146 & 2 & [192,1506] & 6.58 & 4.00 & 0 & [48,50] & [32,47] &  &  &  &  \\
147 & 2 & [192,1508] & 6.58 & 4.00 & 0 & [48,50] & [32,49] &  &  &  &  \\
148 & 2 & [192,1524] & 7.58 & 1.58 & 0 & [96,230] & [6,2] & [3,1] &  &  &
\\
149 & 2 & [192,1525] & 7.58 & 1.58 & 0 & [96,230] & [6,2] & [3,1] &  &  &
\\
150 & 1 & [192,1541] & 7.58 & 6.00 & 0 & [192,1541] & [64,267] &  &  &  &
\\ \hline
&  &  &  &  & & & & & & &%
\end{tabular}%

The following informations is listed in Table V for each isoclinism family;

\begin{enumerate}
\item the number of groups in the family;

\item the rank; the rank of $G$ is $\log _{2}\left\vert Z(G)\cap G^{\prime
}\right\vert +\log _{2}\left\vert G/Z(G)\right\vert .$

\item the middle length; the middle length of $G$ \ is $\log _{2}\left\vert
G^{\prime }/Z(G)\cap G^{\prime }\right\vert .$

\item the nilpotency class, $c>0,$ of the groups in the family.

\item the group id of $G/Z(G)$, of a group $G$ in the family.

\item the group id of the non-trivial or non-repeatedly terms, $\gamma
_{2}(G),...,\gamma _{c}(G),$ of the lower central series of $G;$ here $%
\gamma _{1}(G)=G$ and $\gamma _{i+1}(G)=[\gamma _{i}(G),G]$ for $1\leq i\leq
c.$
\end{enumerate}


\newpage

\begin{thebibliography}{99}
\bibitem[{Alp and Wensley (2013)}]{AW}
Alp, M., \& Wensley, C.D., XMOD: Crossed modules and cat$^{1}$%
-groups in GAP, 2013, \textit{GAP share package} Version 2.26. \vspace{0.4cm}

\bibitem[{Beyl and Tappe (1982)}]{TB}
Beyl, F.R., \& Tappe, J., Group extensions, representations and
the Schur multiplier, 1982, Springer-Verlag. \vspace{0.4cm}

\bibitem[{Brown and Wensley (2003)}]{RB}
Brown, R., \& Wensley, C.D., Computation and homotopical
applications of induced crossed modules, 2003, Journal of Symbolic
Computation, 35, 59-72. \vspace{0.4cm}

\bibitem[{Brown et al. (2011)}]{RB2}
Brown, R., Higgins PJ., \& Sivera, R., Nonabelian algebraic
topology, 2011, European Math. Soc.. \vspace{0.4cm}

\bibitem[{Ellis  (2004)}]{GE}
Ellis, G., Computing group resolutions, 2004, Journal of
Symbolic Computation, 38, 1077-1118. \vspace{0.4cm}

\bibitem[{Ellis and Luyen (2012)}]{GE1}
Ellis, G., \& Luyen, L., Computational homology of \textit{n}-types, 2012,
Journal of Symbolic Computation, 47(11), 1309-1317.\vspace{0.4cm}

\bibitem[{Ellis (2013)}]{GE2}
Ellis, G., HAP: Homological Algebra Programming, \textit{GAP
share package} Version 1.10.15 (2013). \vspace{0.4cm}

\bibitem[{Hall and Senior (1964)}]{MH}
Hall, M., \& Senior, J.K., The groups of order $2^{n}$ $(n\leq
6), $ 1964, Macmillan. \vspace{0.4cm}

\bibitem[{Hall (1940)}]{PH}
Hall, P., The classification of prime power groups, 1940,
Journal f\"{u}r die reine und angewandte Mathematik, 182, 130-141. \vspace{0.4cm}

\bibitem[{Huq (1968)}]{HU}
Huq, S. A., Commutator, nilpotency and solvability in
categories, 1968,  Quart. J. Math. Oxford Ser., 2, 19, 363-389. \vspace{0.4cm}

\bibitem[{James et al. (1990)}]{RJ}
James, R., Newman, M.F., \& O'Brien, E.A., The groups of order
128, 1990, Journal of Algebra, 129, 136-158. \vspace{0.4cm}

\bibitem[{Jones and Wiegold (1974)}]{MRJ}
Jones, M.R., \& Wiegold, J., Isoclinism and covering groups, 1974,
Bull. Austral. Math. Soc., 11, 71-76. \vspace{0.4cm}

\bibitem[{Lue (1979)}]{LU}
Lue, A.S.T., Semi-complete crossed modules and holomorphs of
groups, 1979, The Bulletin of the London Mathematical Society, 11, 8-16.\vspace{0.4cm}

\bibitem[{Modabbernia (2012)}]{RM}
Modabbernia, R., Isologism, Schur-pair property and
Baer-invariant of groups, 2012, World Applied Sciences Journal, 16, 1631-1637. \vspace{0.4cm}

\bibitem[{Mohammadzadeh et al. (2013)}]{HM}
Mohammadzadeh, H., Salemkar, A.R., \& Riyahi, Z., Isoclinic
extensions of Lie algebras, 2013, Turkish Journal of Mathematics, 598-606. \vspace{0.4cm}

\bibitem[{Norrie (1990)}]{NO}
Norrie, K.J., Actions and automorphism of crossed modules, 1990,
Bull. Soc. Math. France, 118, 129-146. \vspace{0.4cm}

\bibitem[{Norrie (1987)}]{NO1}
Norrie, K.J., Crossed modules and analogues of group theorems, 1987,
Ph.D. Thesis, University of London. \vspace{0.4cm}

\bibitem[{Parvaneh et al. (2011)}]{FP}
Parvaneh, F., Moghaddam R., \& Khaksar, A., Some properties of
\textit{n}-isoclinism in Lie algebras, 2011, Italian Journal of Pure and Appl. Math., 28, 165-176. \vspace{0.4cm}

\bibitem[{Porter (2011)}]{TP2}
Porter, T., The crossed menagerie, 2011,http://ncatlab.org/timporter/files/
menagerie11.pdf. \vspace{0.4cm}

\bibitem[{Salemkar et al. (2008)}]{ARS}
Salemkar, A.R., Bigdely H., \& Alamian, V., Some properties on
isoclinism of Lie algebras and covers, 2008, Journal of Algebra and Its
Appl., 7, 507-516. \vspace{0.4cm}

\bibitem[{Salemkar et al. (2007)}]{AFT}
Salemkar, A.R., Saeedi F., \& Karimi, T., The structure of
isoclinism classes of pairs of groups, 2007, Southeast Asian Bulletin of
Mathematics, 31, 1173-1181. \vspace{0.4cm}

\bibitem[{Tappe (1976)}]{JT}
Tappe, J., On isoclinic groups, 1976, Mathematische Zeitschrift, 148, 147-153.
\vspace{0.4cm}

\bibitem[{Vieites and Casas (2002)}]{VC}
Vieites, A.M., \& Casas, J.M., Some results on central extensions
of crossed modules, 2002, Homology, Homotopy and Applications, 4(1), 29-42. \vspace{0.4cm}

\bibitem[{Whitehead (1948)}]{WH}
Whitehead, J.H.C., On operators in relative homotopy groups, 1948,
Ann. Math., 49, 610-640. \vspace{0.4cm}

\bibitem[{Yadav (2008)}]{YA}
Yadav, M.K., On automorphisms of some finite p-groups, 2008,
Proc. Indian Acad. Sci. (Math. Sci.), 118, 1-11. \vspace{0.4cm}
\end{thebibliography}
\end{document}